\documentclass[11pt,a4paper,twoside]{article}
\usepackage{latexsym,amsfonts,amsmath, amsthm, amssymb}
\usepackage{fullpage}
\usepackage{xcolor,bm, bbm}
\usepackage[utf8x]{inputenc}
\usepackage{array}
\usepackage{mathrsfs}
\usepackage{color}

\usepackage{fourier}

\definecolor{col1}{rgb}{0.3.0.4,0.6}

\newcommand{\R}{\mathbb R}
\newcommand{\N}{\mathbb N}

\def\dint{\textup{d}}

\newcommand{\ind}{{\small 1}\!\!1}

\newtheorem{thm}{Theorem}[section]

\newtheorem{lemma}[thm]{Lemma}
\newtheorem{df}[thm]{Definition}

\newtheorem{thmalpha}{Theorem}

{
\theoremstyle{definition}

\newtheorem{rmk}[thm]{Remark}

}
\allowdisplaybreaks
\begin{document}

\title{\bf A Maxwell principle for generalized Orlicz balls}

\medskip

\author{Samuel G.~G. Johnston and Joscha Prochno}

%\thanks{~}

%\keywords{}
%\subjclass{}
%% NB There should be only one primary classification, and zero or
%more secondary classifications.

\date{}

\maketitle

\begin{abstract}
\small
In [A dozen de {F}inetti-style results in search of a theory, Ann. Inst. H. Poincar\'{e} Probab. Statist. 23(2)(1987), 397--423], Diaconis and Freedman studied the low-dimensional projections of random vectors from the Euclidean unit sphere and the simplex in high dimensions, noting that the individual coordinates of these random vectors look like Gaussian and exponential random variables respectively. In subsequent works, Rachev and R\"uschendorf and Naor and Romik unified these results by establishing a connection between $\ell_p^N$ balls and a $p$-generalized Gaussian distribution. In this paper, we study similar questions in a significantly generalized and unifying setting, looking at low-dimensional projections of random vectors uniformly distributed on sets of the form 
\[ B_{\phi,t}^N := \Big\{ (s_1,\ldots,s_N) \in \mathbb{R}^N : \sum_{ i =1}^N \phi(s_i) \leq t N \Big\},\]
where $\phi:\mathbb{R} \to [0,\infty]$ is a function satisfying some fairly mild conditions; in particular, we cover the case of Orlicz functions.
Our method is different from both Rachev-R\"uschendorf and Naor-Romik, based on a large deviation perspective in the form of quantitative versions of Cram\'er's theorem and the Gibbs conditioning principle, providing a natural framework beyond the $p$-generalized Gaussian distribution while simultaneously unraveling the role this distribution plays in relation to the geometry of $\ell_p^N$ balls. We find that there is a critical parameter $t_{\mathrm{crit}}$ at which there is a phase transition in the behaviour of the low-dimensional projections: for $t > t_{\mathrm{crit}}$ the coordinates of random vectors sampled from $B_{\phi,t}^N$ behave like uniform random variables, but for $t \leq t_{\mathrm{crit}}$ however the Gibbs conditioning principle comes into play, and here there is a parameter $\beta_t > 0$ (the inverse temperature) such that the coordinates are approximately distributed according to a density proportional to $e^{ - \beta_t \phi(s) }$.

\medspace
\end{abstract}

%\tableofcontents

% % % % % % % % % % % % % % % % % % % %
% % % % % % % % % % % % % % % % % % % %
% % % % % % % % % % % % % % % % % % % %
% % % % % % % % % % % % % % % % % % % %
\section{Introduction} \label{sec:intro}
% % % % % % % % % % % % % % % % % % % %
% % % % % % % % % % % % % % % % % % % %
% % % % % % % % % % % % % % % % % % % %
% % % % % % % % % % % % % % % % % % % %

% % % % % % % % % % % % % % % % % % % % % % %
%\subsection{Projections of $\ell_p^N$ spheres}
% % % % % % % % % % % % % % % % % % % % % % % 

\subsection{The Maxwell principle}
Over a century ago, Borel \cite[Chapter 5]{borel} observed, independently of Maxwell, that if one chooses a random vector uniformly from the sphere in dimension $N$, then when $N$ is large any given coordinate of the random vector is approximately Gaussian distributed. This result is commonly known as the Maxwell-Borel lemma or the Maxwell principle. More precisely, consider the Euclidean sphere
\begin{align*}
S_2^{N-1} := \left\{ (s_1,\ldots,s_N) \in \mathbb{R}^N \,:\, \sum_{ i = 1}^N s_i^2 = N \right\} 
\end{align*}
of radius $\sqrt{N}$ in $\R^N$, with the normalization here taken to ensure that the typical coordinate of an element of $S_2^{N-1}$ has unit order when $N$ is large. Suppose now $(\Theta_1,\ldots,\Theta_N)$ is a random vector chosen according to $\sigma^N$, the unique rotationally invariant probability measure on $S_2^{N-1}$. For $k < N$, let $\sigma^{N \to k}$ denote the probability density function on $\mathbb{R}^k$ associated with the marginal law of the first $k$ coordinates $(\Theta_1,\ldots,\Theta_k)$ of $(\Theta_1,\ldots,\Theta_N)$, which is given by
\begin{align*}
\sigma^{N \to k}(s_1,\ldots,s_k) := \frac{ \Gamma(\frac{N}{2} )}{  \Gamma(\frac{N-k}{2} ) (\pi N)^{k/2} }  \left( 1 - \frac{ \sum_{ j = 1}^k s_j^2 }{ N } \right)^{ \frac{ N - k - 1}{2} } \ind \{ s_1^2 + \ldots + s_k^2 \leq N \}.
\end{align*}
It turns out that when $N$ is large and $k$ is small compared to $N$, the probability density $\sigma^{N \to k}$ on $\mathbb{R}^k$ is very close to the $k$-dimensional product $\gamma^{\otimes k}$ of the standard Gaussian density 
\begin{align*}
\gamma(s) := \frac{1}{ \sqrt{2 \pi }} e^{ - s^2/2}, \qquad s \in \mathbb{R}.
\end{align*}
(Here and throughout we write $\nu^{\otimes k}$ for the product measure on $\mathbb{R}^k$ associated with a measure $\nu$ on $\mathbb{R}$.)
More specifically, in \cite{DF} Diaconis and Freedman supply the explicit bound 
\begin{align*}
\int_{\mathbb{R}^k } |\sigma^{N \to k}(s) - \gamma^{ \otimes k} (s)  |\, \dint s \leq \frac{ 2(k+3)}{ N - k - 3}
\end{align*}
on the total variation distance between the two probability density functions. 

Diaconis and Freedman go on to observe seemingly connected phenomena in different settings. Indeed, in place of $S_2^{N-1}$ take instead the 
%radius-$N$ 
simplex
\begin{align} \label{eq:simplex}
D^{N-1} := \left\{ (s_1,\ldots,s_N) \in \mathbb{R}^N \,:\, s_i \geq 0,\, \sum_{ i =1}^N s_i = N \right\},
\end{align}
and this time for $k < N$ let $\mu^{N \to k}$ denote the probability density function of the first $k$ coordinates of a random vector selected uniformly from $D^{N-1}$. Diaconis and Freedman also give an explicit bound on the total variation distance between $\mu^{N \to k}$ and the $k$-dimensional product of the standard exponential density $\rho(s) := \ind_{[0,\infty)}(s) e^{ -s}$, though proclaim in their paper to not have the right general theorem unifying these different but seemingly related observations. 

A few years later, Rachev and R\"uschendorf \cite{RR} connected these 
%seemingly related 
phenomena in the setting of the $\ell_p^{N}$ spheres (we write $p$-spheres)
\begin{align*}
S_p^{N-1} := \left\{ (s_1,\ldots,s_N) \in \mathbb{R}^N : \sum_{ i = 1}^N |s_i|^p = N \right\},  \qquad p \geq 1,
\end{align*} 
establishing an $\ell_p$-version of the classical Maxwell-Borel lemma. Indeed, to outline this connection here, consider the standard $p$-Gaussian density
\begin{align*}
\gamma_p( s) := \frac{1}{  2 p^{1/p} \Gamma(1 + 1/p)  } \exp \left( - |s|^p/p \right), \qquad s \in \mathbb{R}.
\end{align*}
Rachev and R\"uschendorf study random vectors in $\mathbb{R}^N$ distributed according to the \emph{cone measure} $\mu_p^N$ on $S_p^{N-1}$, which may be constructed as follows. If $\zeta$ is distributed according to $\gamma_p^{\otimes N}$, then the normalised vector $\zeta/||\zeta||_p$ takes values in the $p$-sphere and is distributed according to $\mu_p^N$. Rachev and R\"uschendorf use this probabilistic construction to analyze the marginal density $\mu_p^{N \to k}$ of the first $k$ coordinates $(X_1,\ldots,X_k)$ of a random vector $(X_1,\ldots,X_N)$ distributed according to the cone measure $\mu_p^N$ (see, e.g., \cite[Proof of Lemma 4]{NR} for a representation of $\mu_p^{N \to k}$). Indeed, they show that when $N \to \infty$ with $k = o(N)$, we have the total variation estimate
\begin{align} \label{eq:RR}
\int_{\mathbb{R}^k } |\mu_p^{N \to k}(s) - \gamma_p^{ \otimes k}(s)  | \,\dint s = \sqrt{ \frac{2}{ \pi e } } \frac{k}{N} + o\left( k/N \right).
\end{align}
We warn the reader that our definition of total variation is twice that of Rachev and R\"uschendorf.

It transpires that the cone measure and surface measure coincide precisely when $p \in\{1,2,\infty\}$, and the result of Rachev and R\"uschendorf \cite{RR} beautifully connects the previously disparate observations of Diaconis and Freedman above: the case $p=2$ connects the Gaussian distribution to the Euclidean sphere $S_2^{N-1}$, and modulo certain symmetries, the $p=1$ case connects the exponential distribution to the simplex $D^{N-1}$. Rachev and R\"uschendorf remarked that an analysis of the $k$-dimensional projections of a random vector from the arguably more natural surface measure $\sigma_p^N$ on $S_p^{N-1}$ would require a different treatment from that of the cone measure $\mu_p^N$. 

With a view to tackling this problem with the surface measure $\sigma_p^N$, Naor and Romik \cite{NR} studied the discrepancy between the cone measure $\mu_p^N$ and the surface measure $\sigma_p^N$, finding a constant $C_p \in (0,\infty)$ such that the total variation distance between the two measures in $N$-dimensions may be bounded above by $C_p/\sqrt{N}$. They used their bound on the total variation between $\sigma_p^N$ and $\mu_p^N$ in conjunction with Rachev and R\"uschendorf's bound \eqref{eq:RR} to show that we have the total variation bound 
\begin{align*}
\int_{\mathbb{R}^k} \big| \sigma_p^{N \to k} (s) - \gamma_p^{ \otimes k}( s) \big| \,\dint s \leq C_p \left( \frac{k}{N} + \frac{1}{ \sqrt{N}} \right)
\end{align*}
on the $k$-dimensional projections of $\sigma_p^{N}$ and the $k$-dimensional product of $p$-Gaussians. 

\subsection{A brief statement of our main result}
While the works of Rachev and R\"uschendorf  \cite{RR} and Naor and Romik \cite{NR} certainly provide a gratifying answer to Diaconis and Freedman's \cite{DF} appeal for a unified theory, the purpose of the present paper is to show that these phenomena may be generalized significantly further in the setting of generalized Orlicz balls
\begin{align} \label{eq:orlicz}
B_{\phi,t}^N := \left\{ (s_1,\ldots,s_N)\in\R^N \,:\, \sum_{ i = 1}^N \phi(s_i) \leq t N \right\},
\end{align}
where $t > 0$ and $\phi:\mathbb{R} \to [0,\infty]$ is a \emph{potential} --- a function satisfying some fairly mild conditions. These conditions are given in Definition \ref{def:potential} below, but let us just say here that with the restrictions imposed on $\phi$, our framework includes the simplex as well as the $\ell_p^N$ balls for $p > 0$ and the more general case of Orlicz balls (where $\phi$ is an even and convex function such that $\phi(0)=0$ and $\phi(t)> 0$ for $t\neq 0$), though in general the sets $B_{\phi,t}^N$ we consider need not be convex, symmetric, simply connected or even compact. These generalisations are made possible through a different perspective based on the large deviation theory and the Gibbs conditioning principle.

We call $D_\phi := \{ s\in\R \,:\, \phi(s) < \infty \}$ the \emph{domain} of $\phi$. 
Giving a very brief outline of our main result here, we find that there is a phase transition in the behaviour of the low-dimensional projections at the critical value $t_{\mathrm{crit}} = t_{\mathrm{crit}}(\phi) \in (0,\infty]$ given by
\begin{align*}
t_{\mathrm{crit}} := 
\begin{cases}
\frac{ 1}{ |D_\phi |} \int_{ D_\phi} \phi(s) \dint  s \qquad &:\text{$|D_\phi| < \infty$},\\
\infty \qquad &:\text{$|D_\phi| = \infty$},
\end{cases}
\end{align*} 
where $|D_\phi|$ is the Lebesgue measure of $D_\phi$. 
More specifically, we have the following:
\begin{itemize}
\item For $t > t_{\mathrm{crit}}$ (so that $|D_\phi| < \infty$), in high dimensions the ball $B_{\phi,t}^N$ is volumetrically similar (in sense that their intersection carries a lot of mass) to the $N$-fold product $D_\phi^N$ of the domain $D_\phi$, so that the density on $\mathbb{R}^k$ associated with the marginal distribution of the first $k$ coordinates of a random vector uniformly distributed on $B^N_{\phi,t}$ is close in total variation distance to the $k$-dimensional product of the uniform density
\begin{align*}
\gamma_{\phi, \mathrm{uni}}(s) := \frac{ \ind_{ D_\phi}(s) }{  |D_\phi| }, \quad s\in\R
\end{align*}  
on $D_\phi$, where $|D_\phi|$ denotes the Lebesgue measure of $D_\phi$.
\item When $t \leq t_{\mathrm{crit}}$, the Gibbs conditioning principle comes into play, and the density on $\mathbb{R}^k$ associated with the distribution of the $k$-dimensional projections of a random vector uniformly distributed on $B^N_{\phi,t}$ are close in total variation distance to the $k$-dimensional product of the Gibbs density
\begin{align*}
\gamma_{\phi, - \beta_t}(s) :=  e^{ - \beta_t \phi(s) } \ind_{D_\phi}(s)/ Z(- \beta_t), \quad s\in\R
\end{align*}  
where $Z(-\beta_t) \in(0,\infty)$ is the normalisation constant known as partition function, and $\beta_t > 0$ is a parameter, in statistical mechanics parlance referred to as inverse temperature, chosen so that
\[\int_{-\infty}^\infty \phi(s) \gamma_{\phi, - \beta_t}(s) \mathrm{d}s = t.\]
Roughly speaking, this suggests that the coordinates of balls $B_{\phi,t}^N$ are distributed according to a density with respect to Lebesgue measure for which larger values of $\phi$ are penalised. The exponential parameter $\beta_t$ increases as $t$ decreases, so that this penalisation becomes stronger as the size of the ball $B_{\phi,t}^N$ shrinks.
\end{itemize}

That concludes our very brief outline of our main results, in the next section we provide a more complete picture, where precise statements on the behaviour of the low-dimensional projections are given in terms of total variation distance.

% % % % % % % % % % % % % % % % % % % % % % % % %
\subsection{Main results} \label{sec:results}
% % % % % % % % % % % % % % % % % % % % % % % % %

We now present our definition for the class of functions $\phi:\mathbb{R} \to [0,\infty]$ determining the generalized Orlicz balls $B_{\phi,t}^N$ we consider. We call such functions potentials.
\begin{df} \label{def:potential}
A measurable function $\phi:\mathbb{R} \to [0,\infty]$ is a \emph{potential} if it is piecewise differentiable, the essential infimum of $\phi$ is zero, for $y < \infty$ the level sets $\{ s  \in\R \,:\, \phi(s) = y \}$ of $\phi$ are finite, and setting $D_\phi := \{ s \in \mathbb{R} : \phi(s) < \infty \}$, the partition function 
\begin{align} \label{eq:pf}
Z( \alpha ) := \int_{D_\phi} e^{ \alpha \phi (s) }\, \dint s 
\end{align} 
is finite and positive for all $\alpha$ in some non-empty interval $(-\infty,\alpha_{\mathrm{max}})$ where $\alpha_{\mathrm{max}} := \sup \{ \alpha \in \mathbb{R} : Z( \alpha ) < \infty \}$. 
\end{df} 

\begin{rmk}
As mentioned above, we call the set $D_\phi = \{ s\in\R \,:\, \phi(s) < \infty \}$ the \emph{domain} of $\phi$, and write $|D_\phi| \in (0,\infty]$ for its Lebesgue measure. Whenever $|D_\phi| = \infty$, clearly $Z(\alpha)$ may only exist for negative values of $\alpha$. We also note that the existence of the partition function $Z(\alpha)$ for some $\alpha$ guarantees that for each $y < \infty$, $\phi^{ - 1} \left( [0,y] \right) := \{ s\in\R \,:\, \phi(s) \in [0,y] \}$ has finite Lebesgue measure, so that $B_{\phi,t}^N$ is a subset of $\phi^{-1} \left( [0,tN] \right)^N$, and hence also has finite Lebesgue measure. In particular, whenever $\phi$ is a potential it always makes sense to say a vector is \emph{uniformly distributed} on $B_{\phi,t}^N$.
\end{rmk}

The generalized Orlicz balls $B_{\phi,t}^N$ include various sets of interest. All subsets of $\mathbb{R}^N$ considered in Section 1.1 are boundaries of generalized Orlicz balls: by taking $\phi(s) = |s|^p$ we recover the $\ell_p^N$ ball of radius $(tN)^{1/p}$, and the function $\phi(s) := \infty \ind_{(-\infty,0)}(s) + s$ corresponds to the simplex $D^{N-1}$ defined in \eqref{eq:simplex}. Moreover, the classical Orlicz balls also fall into this framework, being those sets corresponding to $\phi(0) = 0$, $\phi(t)>0$ for $t\neq 0$, and $\phi$ is even and convex (note that then the set of points of non-differentiability is at most countable); see \cite[Equation (3)]{KP2020}.

Let us emphasise however that we make no further restrictions on $\phi$, so that while the sets $B_{\phi,t}^N$ we consider have finite $N$-dimensional Lebesgue measure and are invariant under permutations of the coordinate axes, as mentioned above they need not be centered, convex, symmetric, simply connected or even compact. For an example of a set with none of these properties, we invite the reader to consider the set $B_{\phi,t}^N$ associated with the potential
\begin{align*}
\phi(s) := 
\begin{cases}
s \qquad &\text{if $s \in [k,k+2^{-k})$ for some $k \in \{1,2,\ldots\}$},\\
\infty \qquad &\text{otherwise}.
\end{cases}
\end{align*}
We now define two key quantities related to $\phi$. The first, the \emph{domain supremum} of $\phi$, to be the essential supremum of $\phi$ on its domain, i.e.
\begin{align*}
t_{\mathrm{sup}} := \inf \Big\{ y \in \mathbb{R} \,:\, \left| \phi^{-1}\big((y,\infty)\big) \right| =0 \Big\} \in (0,\infty],
\end{align*} 
where $|A|$ denotes the Lebesgue measure of a measurable set $A$. 
It is easily verified using the fact that $Z( \alpha ) < \infty$ for some $\alpha$ that $|D_\phi| = \infty$ implies $t_{\mathrm{sup}} = \infty$. Our second quantity, which we already gave above, is the \emph{domain average} of $\phi$, given by 
\begin{align*}
t_{\mathrm{crit}} := 
\begin{cases}
\frac{ 1}{ |D_\phi |} \int_{ D_\phi} \phi(s) \dint  s \qquad &:\text{$|D_\phi| < \infty$},\\
\infty \qquad &:\text{$|D_\phi| = \infty$},
\end{cases}
\end{align*} 
where we emphasise that the integral in the former case may be infinite even when $|D_\phi| < \infty$. Of course $t_{\mathrm{crit}} \leq t_{\mathrm{sup}}$. 

Now for any $\alpha$ such that $Z(\alpha) < \infty$, we may define a \emph{Gibbs} probability density
\begin{align*}
\gamma_{\phi, \alpha } (  s ) :=  \frac{ e^{\alpha \phi(s) }  \ind_{D_\phi}(s)  }{ Z(\alpha) },\qquad s\in\R.
\end{align*}
In the case where $D_\phi$ has finite Lebesgue measure, $Z(0) = |D_\phi| < \infty$, so that we may write 
\[\gamma_{\phi,\mathrm{uni}}(s) := \gamma_{\phi,0}(s) = \frac{ \ind_{D_\phi}(s) }{ |D_\phi|}, \qquad s \in \mathbb{R},\]
for the uniform density on the domain of $\phi$. 

Let us remark that 
\begin{align*}
\frac{ \partial}{ \partial \alpha} \log Z(\alpha) = \frac{ \int_{D_\phi} \phi(s) e^{ \alpha \phi(s) } \dint s }{ \int_{D_\phi}  e^{ \alpha \phi(s) } \dint s }  = \int_{-\infty}^\infty \phi(s) \gamma_{\phi,\alpha}(s) \dint s.
\end{align*}
In particular, $W(\alpha) := \frac{ \partial}{ \partial \alpha} \log Z(\alpha)$ is an increasing function from $(-\infty,\alpha_{\mathrm{max}})$ to $(0,t_{\mathrm{sup}})$ satisfying $W(0) = t_{\mathrm{crit}}$. The following lemma, which follows from Petrov's results \cite{P1965} (see also \cite[Theorem 6.2]{D1954}), guarantees that for every $t \in (0,t_{\mathrm{sup}})$ the existence of a parameter $\alpha_t$ such that the expectation of $\phi$ against $\gamma_{\phi,\alpha_t}(s)\mathrm{d}s$ is equal to $t$.

\begin{lemma} \label{lem:parameter}
For each $t \in (0,t_{\mathrm{sup}})$ there exists a unique parameter $\alpha_t \in (-\infty,\alpha_{\mathrm{max}})$ such that
\begin{align*}
W(\alpha_t) =  \int_{-\infty}^\infty \phi(s) \gamma_{\phi,\alpha_t}(s) \dint s = t.
\end{align*}
\end{lemma}

%\begin{proof}
%See Petrov \textcolor{red}{Petrov citation here}.
%\end{proof}

Clearly since $W$ is increasing and satisfies $W(0) = t_{\mathrm{crit}}$, whenever $t > t_{\mathrm{crit}}$, $\alpha_t > 0$, and $t < t_{\mathrm{crit}}$ implies $\alpha_t < 0$. 

We are now ready to present our main results on the low-dimensional projections of random vectors uniformly sampled from generalized Orlicz balls. Here and below, $\mu^{N \to k}_{\phi,t}$ denotes the marginal density of the first $k$ coordinates $(X_1,\ldots,X_k)$ of a random vector $(X_1,\ldots,X_N)$ chosen according to the uniform measure $\mu^N_{\phi,t}$ on the generalized Orlicz ball $B_{\phi,t}^N$. That is
\begin{align*}
\mu_{\phi,t}^{N \to k} (s_1,\ldots,s_k)  := \frac{1}{ |B_{\phi,t}^N| } \int_{ \mathbb{R}^{N-k}  } \ind_{ B_{\phi,t}^N} (s_1,\ldots,s_k,s_{k+1},\ldots,s_N ) \mathrm{d} s_{k+1} \ldots \mathrm{d} s_N
\end{align*}

With $t_{\mathrm{crit}}$ and $t_{\mathrm{sup}}$ now defined, shortly we state our main result, Theorem \ref{thm:main}. Beforehand however, let us briefly highlight that the case $t \geq t_{\mathrm{sup}}$ is trivial. Here, the generalized Orlicz ball is identical to the $N$-fold product of the domain of $\phi$, that is $B_{\phi,t}^N = D_\phi^N = \left\{ (s_1,\ldots,s_N)\in\R^N \,:\, s_i \in D_\phi \right\}$, and hence the coordinates of $B_{\phi,t}^N$ are independent and uniformly distributed on $D_\phi$. In particular, $\mu^{N \to k }_{\phi,t} = \gamma_{\phi,\mathrm{uni}}^{ \otimes k}$. 
%Theorem \ref{thm:main} concerns the low dimensional projections of $B_{\phi,t}^N$ in the more interesting setting in which $t < t_{\mathrm{sup}}$.}

Our main statement, Theorem \ref{thm:main} (which is actually a condensed statement of more detailed theorems which we state in full in Sections \ref{sec:main2} and \ref{sec:main}), concerns the low-dimensional projections of $B_{\phi,t}^N$ in the more interesting setting in which $t < t_{\mathrm{sup}}$. Here we find that a phase transition occurs at the point $t = t_{\mathrm{crit}}$.

\begin{thmalpha} \label{thm:main}
If $t_{\mathrm{crit}} < t < t_{\mathrm{sup}}$, then the $k$-dimensional projections of a random vector uniformly sampled from $B_{\phi,t}^{N}$ are close in distribution to a $k$-dimensional product of $\gamma_{\phi,\mathrm{uni}}$. More specifically, there exist constants $C = C(\phi,t), c= c(\phi,t) \in (0,\infty)$  depending on $\phi$ and $t$ such that for all $k,N\in\N$ with $k\leq N$, we have 
\begin{align} \label{eq:TV}
\int_{\mathbb{R}^k}  \Big| \mu^{N \to k}_{\phi,t}(s) - \gamma_{\phi,\mathrm{uni}}^{\otimes k }(s) \Big| \, \dint s \leq C e^{ - cN}.
\end{align}
On the other hand, if $ t \leq t_{\mathrm{crit}}$, then the $k$-dimensional projections of a random vector uniformly sampled from $B_{\phi,t}^{N}$ are close in distribution to a $k$-dimensional product of $\gamma_{\phi,\alpha_t}$, where $\alpha_t$ is as in Lemma \ref{lem:parameter}. More specifically, there exists a constant $C = C(\phi,t) \in (0,\infty)$ depending on $\phi$ and $t$ such that for all $k,N\in\N$ with $k\leq N$, we have 
\begin{align} \label{eq:TV2}
\int_{\mathbb{R}^k}  \Big| \mu^{N \to k}_{\phi,t}(s) - \gamma_{\phi,\alpha_t}^{\otimes k }(s) \Big| \, \dint s \leq C \frac{k}{N}.
\end{align}
\end{thmalpha}
In fact, the statement of our main results in Theorem \ref{thm:main} has been abbreviated somewhat for the sake of clarity. In Sections \ref{sec:main2} and \ref{sec:main} the two cases $t_{\mathrm{crit}} < t$ and $t \leq t_{\mathrm{crit}}$ are considered seperately, and in both settings we obtain fine estimates of the total variations occuring on the left-hand sides of \eqref{eq:TV} and \eqref{eq:TV2}. These estimates lead to the rougher bounds given in the statement of Theorem \ref{thm:main}.

\subsection{Further discussion}
Our approach to proving Theorem \ref{thm:main} is based around ideas from the theory of large deviations and statistical mechanics, specifically those centered around Cram\'er's theorem, the Gibbs conditioning principle, and Gibbs measures. Indeed, the framework of Gibbs measures in particular seems to be the natural one and demystifies the appearance of the $p$-Gaussian distribution when taking a probabilistic approach to the geometry of $\ell_p^N$ balls. 
In fact, our approach, which is completely different from \cite{NR} and \cite{RR} and of independent interest, is based around a quantitative version of the Gibbs conditioning principle that appears in a different paper by Diaconis and Freedman \cite{DF2}. Somewhat surprisingly this result was neither cited in \cite{NR} nor \cite{RR} (or the recent paper \cite{KR2018}) even though it already contains ideas towards a unified and generalized theory in the sense of Rachev and R\"uschendorf \cite{RR}. But this exact paper \cite{DF2} of Diaconis and Freedman shall be the starting point for us. Let us also mention that the Gibbs conditioning principle and Gibbs measures have been successfully used to tackle other problems of a geometric flavor using probabilistic methods, e.g., \cite{KP2020} and \cite{KLR2019,KR2018}. 

We now outline briefly how these approaches feature in our analysis. Take the first case $t_{\mathrm{crit}} < t < t_{\mathrm{sup}}$ considered in Theorem \ref{thm:main}. The intuition here is that while $t < t_{\mathrm{sup}}$ ensures that $B_{\phi,t}^N$ is a proper subset of the cube $D_\phi^N$, for each $t$ of this form we have the convergence
\begin{align} \label{eq:vol conv}
\frac{ |B_{\phi,t}^N | }{ |D_\phi^N | } \to 1, \qquad \text{as $N \to \infty$}
\end{align}
in the ratios of the Lebesgue measures of the two sets. Roughly speaking this entails that in high dimensions, $B_{\phi,t}^N$ behaves a lot like the product set $D_\phi^N$, so that the marginal density of the low-dimensional projection is close in total variation to the product of the uniform density on $D_\phi$ in the sense of \eqref{eq:TV}. More specifically, we use a quantitative version of Cramer's theorem from large deviation theory to estimate the discrepancy in volume in the two sets, ultimately showing that it decays exponentially in $N$, leading to the bound in \eqref{eq:TV}.

Now let us consider on the other hand the situation where $t < t_{\mathrm{crit}}$, which we regard as the most interesting case. Here we find that in contrast to \eqref{eq:vol conv}, either $|D_\phi| = \infty$, or even when $|D_\phi| < \infty$, we have
\begin{align} \label{eq:vol conv 2}
\frac{ |B_{\phi,t}^N | }{ |D_\phi^N | } \to 0 , \qquad \text{as $N \to \infty$}.
\end{align}
In any case, here we require a more delicate approach based on the Gibbs conditioning principle, for which we now provide a very rough outline. Suppose $Y_1,Y_2,\ldots$ are independent and identically distributed random variables distributed according to a probability density $\nu$ on the real line, and suppose further that $\mathbb{E}[Y_1 ]  < t_0$. The Gibbs conditioning principle is concerned with the asymptotic distribution of first $k$ coordinates $(Y_1,\ldots,Y_k)$ conditioned on the large deviation event $\{ Y_1 + \ldots + Y_N < t N\}$. The Gibbs conditioning principle states that when $N$ is large and $k$ is small compared to $N$, then under certain conditions these first $k$ coordinates are approximately distributed according to a $k$-dimensional product of the density
\begin{align*}
\nu_{\alpha_t}(s) := e^{ \alpha_t s } \nu(s) / Z(\alpha_t),\quad s\in\R,
\end{align*}
where $Z(\alpha_t)$ is a normalisation constant and $\alpha_t < 0$ is a parameter chosen so that $\int_{ - \infty}^\infty s \nu_{\alpha_t}(s)  \mathrm{d} s$ is equal to $t$. 

With a view to relating the Gibbs conditioning principle to our problem, in the setting where $|D_\phi|$ is finite, we take a sequence of random variables $X_1,X_2,\ldots$ sampled uniformly and independently from $D_\phi$, and consider the transformed sequence $Y_1,Y_2,\ldots$ given by $Y_i := \phi(X_i)$. This transformation allows us to a express membership of a random vector $(X_1,\ldots,X_N)$ in the subset $B_{\phi,t}^N$ of $D_\phi^N$ in terms of the large deviation event $\{ \phi(X_1) + \ldots + \phi(X_N) \leq  N t \} = \{ Y_1 + \ldots + Y_N \leq N t \}$, and therefore use the Gibbs conditioning principle to understand the low-dimensional projections of $(Y_1,\ldots,Y_k)$, and hence ultimately $(X_1,\ldots,X_k)$ on this event. It transpires that we are able to adapt our proof to ultimately make sense of sampling uniformly from $D_\phi$ even when the Lebesgue measure of $D_\phi$ is infinite, and by using a quantitative version of the Gibbs conditioning principle due to Diaconis and Freedman, we obtain the bound \eqref{eq:TV2}. Finally let us mention that in our treatment we regard the case $t = t_{\mathrm{crit}}$ as fitting into the framework of Gibbs conditioning in correspondence with the parameter $\alpha_t = 0$. In particular, the convergence in total variation in this critical case happens at the linear \eqref{eq:TV2} rather than exponential rate \eqref{eq:TV}.
\vspace{5mm}

We close the introduction by taking a moment to clarify a small difference between the statement of Theorem \ref{thm:main} with the function $\phi(s) = s^p$ and the frameworks considered by Diaconis and Freedman \cite{DF}, Rachev and R\"uschendorf \cite{RR}, and Naor and Romik \cite{NR}. Namely, the above authors consider the boundary of the $\ell_p^N$ ball, while we consider the interior of such sets. Ultimately we recover the same result despite the slightly different setting, namely that the low-dimensional projections are approximately $p$-Gaussian. 
Of course, let us also note that in many situations there is no difference between the results for the uniform distribution or the distribution with respect to the cone probability measure, even though sometimes different methods are required (see, e.g., \cite{HPT2018, KPT2019_I, KPV2020, PTT2019, SchechtmanZinn}).

That concludes the introduction. We now take a moment to overview the remainder of the paper.
\subsection{Overview}
The remainder of the paper is structured as follows:
\begin{itemize}
\item In Section \ref{sec:tech} we study the stability of certain Gibbs measures under truncations of their tail mass, proving stability of these truncations under various metrics, including large deviation and moments. Our work in this section allows us to consider potentials $\phi$ with non-compact domains $\{ x : \phi(x) < \infty \}$ so that one cannot sample `uniformly' from the domain of $\phi$ as we have done above in the proof sketch above.
\item In Section \ref{sec:sufficiency}, we study how total variation distances are preserved under pushforwards and pullbacks, giving us the tools necessary to study potentials that are many-to-one. 
\item In Section \ref{sec:qgibbs0} we supply quantitative statements of Cramer's theorem and the Gibbs conditioning principle that are used in the following two sections to prove Theorem \ref{thm:main}. 
\item In Section \ref{sec:main2} we give a proof of the Theorem \ref{thm:main} in the case $t > t_{\mathrm{crit}}$. 
\item In Section \ref{sec:main} we give a proof of Theorem \ref{thm:main} in the case $t \leq t_{\mathrm{crit}}$.
\item In the appendix of the paper we give a proof sketch of the quantitative version of the Gibbs conditioning principle, Lemma \ref{lem:qgibbs}. 
\end{itemize}

% % % % % % % % % % % % % % % % % % % %
% % % % % % % % % % % % % % % % % % % %
% % % % % % % % % % % % % % % % % % % %
% % % % % % % % % % % % % % % % % % % %
\section{Total variation and Gibbs truncations} \label{sec:tech}
% % % % % % % % % % % % % % % % % % % %
% % % % % % % % % % % % % % % % % % % %
% % % % % % % % % % % % % % % % % % % %
% % % % % % % % % % % % % % % % % % % %

The main task of this section is the statement and proof of Lemma \ref{lem:tech} below concerning asymptotics of certain truncations of measures. First, we start with a quick lemma on the tail moments of probability measures with exponential moments.
\begin{lemma} \label{lem:momenttails}
Suppose $f:[0,\infty) \to [0,\infty)$ is a function such that there exist constants $c,C\in(0,\infty)$ such that
\begin{align*}
f(s) \leq C e^{ - cs}.
\end{align*} 
Then, for all $\kappa \geq 0$, and all $L > \kappa /c$, we have 
\begin{align*}
\int_L^\infty s^\kappa f(s) \, \dint s \leq \frac{ C L^\kappa e^{ - c L} }{ c - \kappa /L} \leq C' e^{ - c'L}
\end{align*}
for constants $C',c'\in(0,\infty)$ depending on $\kappa$ but independent of $L$.
\end{lemma}
\begin{proof}
We have 
\begin{align*}
\int_L^\infty s^\kappa f (s) \, \dint s  \leq  C \int_L^\infty s^\kappa e^{ - cs} \, \dint s = C L^\kappa e^{ - cL} \int_0^\infty \left( 1 + \frac{s}{L} \right)^\kappa e^{ - c s} \, \dint s .
\end{align*} 
Now use the bound $ \left( 1 + \frac{s}{L} \right)^\kappa  \leq e^{ \kappa s / L}$. 
\end{proof}

Before stating Lemma \ref{lem:tech}, we give a brief informal statement. Suppose we have a probability measure $\nu$ on $(0,\infty)$ with exponential moments, and we create a truncated version $\nu_{\langle L \rangle}$ of the measure restricted to taking values in $[0,y)$, but exponentially tilted so that $\nu_{\langle L \rangle}$ has the same mean as $\nu$. The following lemma states that the measure $\nu_{\langle L \rangle}$ is stable under tail truncations, in that when $L$ is large, it has similar moments to $\nu$ and is close in total variation distance to $\nu$. 

\begin{lemma} \label{lem:tech}
Let $\nu:[0,\infty)\to[0,\infty)$ be a probability density satisfying $\nu(s) \leq C e^{ - c s}$ for all $s \geq 0$ and with expectation $\int_0^\infty s \nu(s) \,\dint s = t$.  Given $L > t$, we define the tilted truncation $\nu_{\langle L \rangle}$ of $\nu$ to be the probability density function
\begin{align*}
\nu_{\langle L \rangle} (s) := \frac{ e^{\alpha s} \ind_{\{s < L\}} \nu(s)  }{ \int_0^L e^{\alpha s } \nu(s) \, \dint s } ,
\end{align*} 
on $[0,\infty)$, where $\alpha = \alpha(L)$ is the unique parameter chosen so that $\int_0^L s \nu_{\langle L \rangle} (s) \, \dint s = t$. 

Then there exist constants $C',c', L_0 \in(0,\infty)$ such whenever $L \geq L_0$, we have the following bounds:
\begin{enumerate}
\item The size of $\alpha(L)$ is bounded by $\alpha(L) \leq C' e^{ - c' L }$.
\item The total variation between $\nu_{\langle L \rangle}$ and $\nu$ is bounded by 
\begin{align*}
||  \nu_{ \langle L \rangle} - \nu || := \int_0^\infty  \left| \nu_{\langle L \rangle}(s) - \nu(s) \right| \, \dint s \leq C' e^{ - c 'L}.
\end{align*}
\item The difference between the $j^{\text{th}}$ moment of $\nu$ and $\nu_{\langle L \rangle}$ is bounded by 
\begin{align*}
\left| \int_0^\infty s^j \nu(s) \, \dint s - \int_0^\infty s^j \nu_{ \langle L \rangle } (s)\, \dint s  \right| \leq \frac{ C'_0 }{ 1 - j/L_0 } L^{j+1 } e^{ - c'_0 L} .
\end{align*} 
\end{enumerate}

\end{lemma}

\begin{proof} 
In the case that $\nu$ is supported on $[0,L]$, $\alpha(L) = 0$ and the measure $\nu_{\langle L \rangle}$ is identical to $\nu$, so that all three statements are immediate. For the remainder of the proof we therefore assume without loss of generality that $\nu$ is not supported on $[0,L]$. 

We begin by showing that when $L$ is large, $\alpha(L)$ is small. First we note that the generating function $G(\alpha) := \int_0^\infty e^{\alpha s}\, \nu( s) \mathrm{d}s$ exists in a subset that contains $(-\infty, c)$ and satisfies $G(0) = 1, G'(0) = t$, with $G''(0) > t^2 $ since $\nu$ is non-degenerate. 

Define the function
\begin{align} 
Q(\alpha,L) := \frac{ \int_0^L s e^{\alpha s} \ind_{\{s < L\}}  \nu(\dint s) }{ \int_0^L e^{\alpha s } \nu(\dint s) } .
\end{align}
Then since $\nu$ is non-degenerate, $Q(\alpha,L)$ is monotone increasing in the $\alpha$ variable, and hence $\alpha(L)$ is the unique solution to the equation $Q\left( \alpha(L) , L \right) = t$. Since $\nu$ is not supported on $[0,L]$, we have $Q(\alpha, L ) < t$ and therefore $\alpha(L) > 0$. 

We now develop a lower bound for $Q(\alpha,L)$. Extending the range of the integral in the denominator to obtain the first inequality below, and then using the fact that $G(\alpha) \geq 1$ whenever $\alpha \geq 0$ to obtain the second, we have
\begin{align} \label{eq:giraffe}
Q(\alpha, L) &= \frac{ \int_0^\infty s e^{\alpha s}   \nu(\dint s) - \int_L^\infty s e^{\alpha s} \nu(\dint s)  }{ \int_0^L e^{\alpha s } \nu(\dint s)  } \nonumber \\
 &\geq \frac{ \int_0^\infty s e^{\alpha s}   \nu(\dint s) - \int_L^\infty s e^{\alpha s} \nu(\dint s)  }{ \int_0^\infty e^{\alpha s } \nu(\dint s)  } \nonumber \\
 &\geq K(\alpha)  - \int_L^\infty s e^{ \alpha s} \nu(\dint s),
\end{align}
where $K(\alpha) := \frac{ \partial}{ \partial \alpha} \log G(\alpha)$. With a view to bounding $Q(\alpha,L)$ below, we now control the two quantities appearing on the right-hand-side of \eqref{eq:giraffe}. 

First we note that $K(0) = t$. Moreover, $K'(0) > 0$ since $G''(0) > G'(0)^2$, so that in particular, there exists a $\delta >0$ and a $\rho > 0$ such that for all $\alpha \in [0,\delta]$, 
\begin{align} \label{eq:vogel2}
K(\alpha) \geq t + \rho \alpha.
\end{align}
As for the integral in the final line of \eqref{eq:giraffe}, recall that $\nu(s) \leq C e^{ - c s}$ for all $s$. In particular, setting $f(s) = e^{\alpha s} \mu(s)$ in Lemma \ref{lem:momenttails}, we see that there are constants $c_1,C_1$ such that for all $\alpha \in [0,c/2]$, 
\begin{align} \label{eq:vogel3}
\int_L^\infty s e^{ \alpha s} \nu(\dint s)  \leq C_1 e^{ - c_1 L }.
\end{align}
Combining \eqref{eq:vogel2} and \eqref{eq:vogel3} in \eqref{eq:giraffe}, we obtain
\begin{align*}
Q(\alpha,L) \geq t + \rho \alpha - C_1 e^{ - c_1 L} \qquad \text{for $\alpha \in [0, \delta \wedge c/2]$}.
\end{align*}
Now let $L_0$ be sufficiently large so that for all $L \geq L_0$ we have $\frac{C_1}{\rho} e^{ - c_1 L } \leq \delta \wedge c/2$. Then plainly, for all $L \geq L_0$, we have 
\begin{align}
Q(\alpha,L) \geq t + \rho \alpha - C_1 e^{ - c_1 L} \qquad \text{for $\alpha \in \left[0, \frac{C_1}{\rho} e^{ - c_1 L } \right]$}.
\end{align}
It follows that for all $L \geq L_0$, $Q\left( \frac{C_1}{\rho} e^{ - c_1 L } , L \right) \geq t$, and hence $\alpha(L) \leq   \frac{C_1}{\rho} e^{ - c_1 L }$, establishing the first statement of Lemma \ref{lem:tech}.  Throughout the remainder of the proof we assume that $L \geq L_0$. 

We now turn to proving the second point concerning the total variation distance between the measure and its tilted truncation. Since $\nu_{\langle L \rangle}$ is supported on $[0,L]$, we have 
\begin{align} \label{eq:fulton}
\int_0^\infty  \left| \nu_{\langle L \rangle}(s) - \nu(s) \right| \, \dint s = \int_0^L \left| \frac{ e^{ \alpha(L) s}  }{ \int_0^L e^{ \alpha(L) s} \nu(s) \dint s }   - 1 \right| \nu(s) \dint s + \int_L^\infty \nu(s) \dint s. 
\end{align} 
First we note that by Lemma \ref{lem:momenttails} there are constants $C_2,c_2\in(0,\infty)$ so that the latter integral on the right-hand side of \eqref{eq:fulton} may be bounded by
\begin{align} \label{eq:cologne}
 \int_L^\infty \nu(s) \dint s \leq C_2 e^{ - c_2L} .
\end{align} 
We turn to bounding the former integral on the right hand side of \eqref{eq:fulton}. To this end, we note that since $\alpha(L) \leq \frac{C_1}{\rho} e^{ - c_1L}$, there are constant $C_3,c_3\in(0,\infty)$ such that for all $s \in [0,L]$ we have
\begin{align} \label{eq:cologne2}
1 \leq e^{\alpha(L)s} \leq e^{  \frac{C_1}{\rho} L e^{ - c_1 L} } \leq 1 + C_3 e^{ -c_3 L}.
\end{align} 

On the other hand, again using Lemma \ref{lem:momenttails} and the fact that $e^{\alpha(L)s } \geq 1$ to obtain the lower bound below, and the upper bound in \eqref{eq:cologne2} to obtain the upper bound below, it may be seen that there are constants $C_4,c_4\in(0,\infty)$ such that
\begin{align} \label{eq:cologne3}
 1 - C_4e^{ - c_4 L} \leq  \int_0^L e^{ \alpha(L) s} \nu(s) \dint s \leq 1 +  C_4e^{ - c_4 L}.
\end{align}
Combining \eqref{eq:cologne2} with \eqref{eq:cologne3}, we see that there exist constants $C_5,c_5 \in (0,L)$ such that for every $s \in [0,L]$,
\begin{align} \label{eq:cologne3b}
\left| \frac{ e^{ \alpha(L) s}  }{ \int_0^L e^{ \alpha(L) s} \nu(s) \dint s }   - 1 \right| \leq C_5 e^{ - c_5 L}.
\end{align}
Using the fact that $\nu$ is a probability measure, \eqref{eq:cologne3b} entails
\begin{align} \label{eq:cologne4}
  \int_0^L  \left| \frac{ e^{ \alpha(L) s}  }{ \int_0^L e^{ \alpha(L) s} \nu(s) \dint s }   - 1 \right| \nu(s) \dint s \leq C_5 e^{ - c_5L}.
\end{align}
In particular, using the bounds \eqref{eq:cologne} and \eqref{eq:cologne4} in \eqref{eq:fulton}, we obtain the second point of the lemma.

It remains to prove the third point concerning the difference between moments of $\nu$ and its tilted truncation $\nu_{\langle L \rangle}$, the proof of which follows quickly from the bound \eqref{eq:cologne3b}. Indeed, using the triangle inequality to obtain the first inequality below, and \eqref{eq:cologne3b} to obtain the second we have 
\begin{align*} 
\left| \int_0^\infty s^j \nu(s) \dint s - \int_0^\infty s^j \nu_{ \langle L \rangle } (s) \dint s \right| &\leq \int_0^L \left| \frac{ e^{ \alpha(L) s}  }{ \int_0^L e^{ \alpha(L) s} \nu(s) \dint s }   - 1 \right| s^j \nu(s) ds + \int_L^\infty s^j \nu(s) \dint s. \nonumber \\
&\leq C_5 e^{ - c_5L}  \int_0^L s^j \nu(s) ds + \int_L^\infty s^j \nu(s) \dint s\\
&\leq C_5 L^j e^{ - c_5 L } + C_6 e^{ - c_6 L},
\end{align*}  
where the final inequality above follows from the fact that $\nu$ is a probability measure to deal with the first term, and an application of Lemma \ref{lem:momenttails} to handle the second. It particular there exist $C_7,c_7\in(0,\infty)$ such that 
\begin{align*} 
\left| \int_0^\infty s^j \nu(s) \dint s - \int_0^\infty s^j \nu_{ \langle L \rangle } (s) \dint s \right| &\leq C_7 e^{ -c_7 L},
\end{align*} 
completing the proof of the third statement in Lemma \ref{lem:tech}.
\end{proof}

%%%%%%%%%%%%%%%%%%%%%%%
%%%%%%%%%%%%%%%%%%%%%%%
%%%%%%%%%%%%%%%%%%%%%%%
\section{Total variation, pushforwards and potentials} \label{sec:sufficiency}
%%%%%%%%%%%%%%%%%%%%%%%
%%%%%%%%%%%%%%%%%%%%%%%
%%%%%%%%%%%%%%%%%%%%%%%

In this section, and throughout the paper, we use the following definition.

\begin{df} \label{def:proj}Whenever $\pi$ is a probability density on $\mathbb{R}^N$, we write $\pi^{N \to k}$ for the marginal density on $\mathbb{R}^k$ of the first $k$ coordinates $(X_1,\ldots,X_k)$, where $(X_1,\ldots,X_N)$ is a random vector in $\mathbb{R}^N$ distributed according to $\pi(s) \mathrm{d}s$.
\end{df}

\subsection{Total variation}

In this section we collect several results on how the total variation metric interacts with product measures and pushforwards. While most of the tools we develop in the section are well known, we have included proofs with a view towards completeness. 

When $\pi$ and $\lambda$ are measures on a measurable space $(E,\mathcal{E})$, we write
\begin{align*}
|| \pi - \lambda || := 2 \sup_{ A \in \mathcal{E} } | \pi(A) - \lambda(A) | 
\end{align*}
for the total variation distance between $\pi$ and $\lambda$. Suppose that $\pi$ and $\lambda$ are both absolutely continuous with respect to a measure $\nu$, so that by the Radon-Nikodym theorem, we have $\dint\pi = f \dint\nu$ and $\dint\lambda = g \dint\nu$ for some measurable functions $f,g:E \to [0,\infty]$. Then it is easily verified that we have the alternative integral representation
\begin{align} \label{eq:japan}
|| \pi - \lambda ||= \int_E | f - g | \dint \nu
\end{align}
for the total variation distance between $\pi$ and $\lambda$. 

\subsection{Total variation, product measures, containments and projections}
The following simple lemma on the total variation distances between product measures is well known, dating back at least as far as Blum and Pathak \cite{BP}. We omit a proof.

\begin{lemma} [Total variation summing lemma] \label{lem:TVsum}
Let $\pi$ and $\lambda$ be probability measures on $(E,\mathcal{E})$, and let $\pi^{\otimes k}$ and $\lambda^{\otimes k}$ be their respective product measures on the $k$-dimensional product space $(E^k, \mathcal{E}^{ \otimes k })$. Then 
\begin{align*}
|| \pi^{ \otimes k } - \lambda^{ \otimes k } || \leq k || \pi - \lambda ||.
\end{align*}
\end{lemma}
Suppose $\pi$ and $\lambda$ are probability measures on $(E,\mathcal{E})$ and $\mathcal{F}$ is a sub-$\sigma$-algebra of $\mathcal{E}$. We write 
\begin{align*}
|| \pi - \lambda ||_{ \mathcal{F} } := 2 \sup_{ F \in \mathcal{F} } | \pi(F) - \lambda(F) | .
\end{align*}

Our next lemma states that if two measures on $\mathbb{R}^N$ are close in total variation, so are their projections.

\begin{lemma} \label{lem:TV projection}
Let $\pi$ and $\lambda$ be probability densities on $\mathbb{R}^N$. Then with $\pi^{N \to k}$ and $\lambda^{N \to k}$ as in Definition \ref{def:proj} we have 
\begin{align*}
\int_{ \mathbb{R}^k } \left| \pi^{N \to k}(s) - \lambda^{N \to k}( s) \right| \dint s \leq \int_{ \mathbb{R}^N } \left| \pi(s) - \lambda(s) \right| \dint s.
\end{align*}
\end{lemma}
\begin{proof}
This is a straightforward application of the triangle inequality. Indeed, for $s = (s_1,\ldots,s_k) \in \mathbb{R}^k$ and $\zeta = (\zeta_1,\ldots,\zeta_{N-k}) \in \mathbb{R}^{N- k}$, write $(s,\zeta) := (s_1,\ldots,s_k,\zeta_1,\ldots,\zeta_{N-k} ) \in \mathbb{R}^N$. Then
\begin{align*}
\int_{ \mathbb{R}^k } \left| \pi^{N \to k}(s) - \lambda^{N \to k}( s) \right| \dint s &= \int_{ \mathbb{R}^k } \left| \int_{ \mathbb{R}^{N-k} } \pi(s,\zeta) - \lambda( s,\zeta) ~ \dint \zeta \right| \dint s  \\
&\leq \int_{ \mathbb{R}^k }  \int_{ \mathbb{R}^{N-k} } \left| \pi(s,\zeta) - \lambda( s,\zeta) \right|   \dint \zeta\dint s  \\
&=  \int_{ \mathbb{R}^N } \left| \pi(s) - \lambda(s) \right| \dint s,
\end{align*}
which completes the proof.
\end{proof}

Our next lemma states that if $A \subseteq B$, and $B \setminus A$ is small, then the projections of uniform random vectors from $A$ and $B$ are close in total variation distance.

\begin{lemma} \label{lem:TV projection 2}
Let $A \subseteq B$ be measurable subsets of $\mathbb{R}^N$ with finite Lebesgue measure, and let $\mu_A$ anud $\mu_B$ denote the uniform densities on $A$ and $B$. Then with $\mu_A^{N \to k}$ and $\mu_B^{N \to k}$ defined using Definition \ref{def:proj}, we have
\begin{align*}
\int_{ \mathbb{R}^k } \Big| \mu_A^{N \to k} (s) - \mu_B^{ N \to k }(s)  \Big| \dint s \leq 2 \frac{ | B \setminus A| }{ |B|}.
\end{align*} 

\end{lemma}
\begin{proof}
Using Lemma \ref{lem:TV projection} to obtain the inequality below, we have
\begin{align*}
\int_{ \mathbb{R}^k } \Big| \mu_A^{N \to k} (s) - \mu_B^{ N \to k }(s)  \Big| \dint s \leq \int_{ \mathbb{R}^N } \Big| \mu_A(s) - \mu_B(s)  \Big| \dint s  =  2 \frac{ | B \setminus A| }{ |B|},
\end{align*}
completing the proof.
\end{proof}

%Recall that if $\mathcal{F}$ is a sub-$\sigma$-algebra of $\mathcal{E}$ and $E' \in \mathcal{E}$, the conditional %expectation $\mu( E' | \mathcal{F}) $ of $E$ given $\mathcal{F}$ is the ($\mu$-a.s. unique) $\mathcal{F}$-%measurable random variable such that for every $F' \in \mathcal{F}$,
%\begin{align*}
%\mu( E' \cap F' ) = \int_{F'} \mu( E' |\mathcal{F} ) \dint \mu
%\end{align*}
%We say a sub-$\sigma$-algebra $\mathcal{F}$ of $\mathcal{E}$ is \emph{sufficient} for $\pi$ and $\lambda$ if $%\pi(E' |\mathcal{F})=  \lambda(E' | \mathcal{F})$ for every $E' \in \mathcal{E}$.
%We now state a lemma due to Diaconis and Freedman \cite{DF}
%\begin{lemma}
%Let $\pi$ and $\lambda$ be probability measures on a measurable space $(E,\mathcal{E})$, and suppose $\mathcal{F}$ is a sub-$\sigma$-algebra of $\mathcal{E}$ that is sufficient for $\pi$ and $\lambda$. Then
%\begin{align*}
%|| \pi - \lambda ||_{ \mathcal{F} }  = || \pi - \lambda ||_{ \mathcal{E} }  
%\end{align*}
%\end{lemma}

% % % % % % % % % % % % % % % % % % % % % % % %
\subsection{Total variation and pushforwards}
% % % % % % % % % % % % % % % % % % % % % % % %

We now relate total variation distances under pushforwards of certain probability measures. Let $(E,\mathcal{E})$ and $(F,\mathcal{F})$ be measurable spaces. If $\mu$ is a measure on $E$, and $\Phi:E\to F$ is a measurable function, we write $\Phi^{ \# } \mu$ for the pushforward measure on $F$, defined by 
\begin{align*}
\Phi^{ \# } \mu \left( B \right) := \mu \left( \Phi^{ - 1}(B) \right)
\end{align*}
for measurable subsets $B$ of $F$. 
If $\mu$ is a probability measure on $E$, and $X$ is a random variable distributed according to $\mu$, then $\Phi^\# \mu$ is the law of $\Phi(X)$. 
It is straightforward to check that if $\mu$ is absolutely continuous with respect to $\lambda$, then $\Phi^\# \mu$ is absolutely continuous with respect to $\Phi^\# \lambda$. 

The following lemma may be regarded as a more measure-theoretic formulation of Diaconis and Freedman's sufficiency lemma, \cite[Lemma (2.4)]{DF}, stating that that total variation of certain measures is preserved under pushforwards.

\begin{lemma} \label{lem:pushforwards}
Let $E$ and $F$ be measurable spaces, suppose $\mu$ is a measure on $E$ and suppose further that $\Phi:E \to F$ is a measurable function. Suppose that $\pi$ and $\lambda$ are probability measures on $E$ that are absolutely continuous with respect to $\mu$, and such that there exist $f,g : F \to [0,\infty]$ such that
\begin{align*}
\frac{ \mathrm{d} \pi }{ \mathrm{d} \mu } = f \circ \Phi \qquad \text{and} \qquad \frac{ \mathrm{d} \lambda }{ \mathrm{d} \mu } = g \circ \Phi. 
\end{align*}
Then 
\begin{align} \label{eq:TV identity}
|| \pi - \lambda || = || \Phi^{\#} \pi - \Phi^\# \lambda ||.
\end{align}
\end{lemma}

\begin{proof}
It is easily verified that 
\begin{align*}
\frac{ \mathrm{d} \Phi^{\# } \pi }{ \mathrm{d} \Phi^{\# } \mu } = f  \qquad \text{and} \qquad \frac{ \mathrm{d}  \Phi^{\# }\lambda }{ \mathrm{d}  \Phi^{\# }\mu } = g. 
\end{align*}
In particular, using \eqref{eq:japan} to obtain the outer equalities below, and changing variable to obtain the central equality, we have 
\begin{align*}
|| \Phi^{\#} \pi - \Phi^\# \lambda || = \int_F | f(y) - g(y) | \Phi^{\# } \mu ( \mathrm{d} y) = \int_E | f( \Phi(s) ) - g( \Phi(s)) | \mu( \mathrm{d} s ) = || \pi - \lambda || 
\end{align*} 
as required.
\end{proof}

% % % % % % % % % % % % % % % % % % % % % % % % % % % % % % % % % %
\subsection{The pushforward by a potential} \label{sec:potential}
% % % % % % % % % % % % % % % % % % % % % % % % % % % % % % % % % %
We will occasionally abuse notation in the following sense: if $\nu$ is a probability density on $\mathbb{R}$ and $f:\mathbb{R} \to \mathbb{R}$ is a measurable mapping, we write  $f^{\#} \nu$ for the probability density on $\mathbb{R}$ associated with the pushforward by $f$ of the measure $\nu(s) \mathrm{d}s$. Now given our potential $\phi$ and a probability measure on $\mathbb{R}$, we would like to understand the densities associated with pushforwards using $\phi$. To this end, consider the increasing function $F:[0,\infty) \to [0,\infty)$ given by letting $F(y)$ denote the Lebesgue measure of the set of all points $s \in \mathbb{R}$ for which $\phi(s) \leq y$, that is $F(y) := \left| \phi^{-1}[ 0, y] \right|$. Suppose $\phi$ is differentiable at $s$ for all $s \in \phi^{-1}(y)$. Then it is easily verified that
\begin{align} \label{eq:psi def}
\psi(y) := F'(y) = \sum_{ s \in \phi^{ -1 }(y) } 1/ | \phi'(s) | ,
\end{align} 
with the understanding that $\psi(y)$ is equal to $+ \infty$ whenever there is an $s \in \phi^{ - 1}(y)$ such that $\phi'(s) = 0$. The function $\psi$ is defined for almost-all $y \in [0,\infty)$, and has the property that for all $f$ such that $f(\phi(s))$ is integrable,
\begin{align} \label{eq:change}
\int_{- \infty}^\infty f\left( \phi(s) \right) \dint s = \int_0^\infty f(y) \psi(y) \dint y.
\end{align}
In particular, whenever $\nu:\mathbb{R} \to [0,\infty)$ is a probability density of the form $\nu(s) = f(\phi(s))$, the pushforward $\phi^{\#}\nu$ of the measure $\nu(s) \mathrm{d}s$ has density $\phi^{\#} \nu(y) := f(y) \psi(y)$.

Recall the partition function $Z(\alpha) := \int_{D_\phi} e^{ \alpha \phi(s) } \dint s$ defined in Section \ref{sec:results}. We note that by \eqref{eq:change} we may alternatively write
\begin{align} \label{eq:new PF}
Z(\alpha) := \int_0^\infty e^{ \alpha y} \psi( y) \dint y.
\end{align}
Now for all $\alpha$ such that $Z(\alpha) < \infty$ we may define the $\alpha$-tilted probability density on $[0,\infty)$ by
\begin{align} \label{eq:alpha push}
\psi_\alpha(y) := \frac{ e^{\alpha y} \psi(y) }{ Z(\alpha) }.
\end{align}
In particular, in the setting where the Lebesgue measure of $D_\phi$ is finite so that $Z(0) = |D_\phi| < \infty$, whenever $X$ is uniformly distributed on $D_\phi$, the random variable $\phi(X)$ is distributed according to the probability density
\begin{align} \label{eq:zero push}
\psi_0(y) := \frac{ \psi(y) }{ |D_\phi|}.
\end{align}
Note that $\psi_\alpha = \phi^{ \# } \gamma_{\phi,\alpha}$ where $\gamma_{\phi,\alpha}$ was defined in Section \ref{sec:results}.

Finally, define the multivariate potential $\Phi:\mathbb{R}^k \to \mathbb{R}^k$ by $\Phi(s_1,\ldots,s_k) := \left( \phi(s_1),\ldots,\phi(s_k) \right)$. We note that whenever $\nu$ is a measure on $\mathbb{R}$, we have
\begin{align*}
( \phi^{ \#} \nu )^{ \otimes k} = \Phi^{\#} ( \nu^{ \otimes k }).
\end{align*}

% % % % % % % % % % % % % % % % % % % %
% % % % % % % % % % % % % % % % % % % %
% % % % % % % % % % % % % % % % % % % %
% % % % % % % % % % % % % % % % % % % %
\section{The quantitative Cram\'er Theorem and Gibbs conditioning principle } \label{sec:qgibbs0}
% % % % % % % % % % % % % % % % % % % %
% % % % % % % % % % % % % % % % % % % %
% % % % % % % % % % % % % % % % % % % %
% % % % % % % % % % % % % % % % % % % %
 
In this section we provide further background on both Cram\'er's theorem and the Gibbs conditioning principle, ultimately giving quantitative versions of both principles that are used in the proof of Theorem \ref{thm:main}. 

To this end, let $Y_1,Y_2, \ldots$ be a sequence of independent random variables identically distributed according to a probability density $\psi$ on $\mathbb{R}$, and suppose that $\mathbb{E}[Y_1]=t_0$. Suppose further that the moment generating function $Z(\alpha) := \int_{-\infty}^\infty e^{\alpha y } \psi(y) \mathrm{d}y$ associated with the density exists in an open interval containing the origin. Fix $k \in\N$ and let $t > t_0$. 

According to Cramer's theorem, \cite[Section 2.2]{DZ2010}, we have
\begin{align} 
\lim_{N \to \infty} \frac{1}{N } \log \mathbb{P} \left( X_1 + \ldots + X_N > t N \right) = I(t),
\end{align}
where $I:[t_0,\infty) \to [0,\infty]$ is a \emph{rate function} given by 
\begin{align*}
I(t) := \alpha_t t - \log Z(\alpha_t),
\end{align*}
where  $\alpha_t$ is the solution to $\frac{ \partial}{ \partial \alpha} \log Z(\alpha) = t$. 

In the present paper we will appeal to a quantitative version of Cram\'er's theorem. Setting 
\[\sigma_t^2 := \frac{ \partial^2}{ \partial \alpha^2} \log Z(\alpha )\big|_{\alpha = \alpha_t},\]
we have the following.

\begin{lemma}[Theorem 3 of Petrov \cite{P1965}] \label{lem:qcramer}
We have 
\begin{align} \label{eq:q cramer}
\mathbb{P} \left( Y_1 + \ldots + Y_N > t N \right) = (1 + \varepsilon_N ) \frac{ 1}{ \sqrt{ 2 \pi \sigma_t^2 N } } e^{ - N I(t) },
\end{align}
where there exists a constant $C\in(0,\infty)$ such that $|\varepsilon_N| \leq C/ \sqrt{N}$. 
\end{lemma}

We now turn to discussing the Gibbs conditioning principle \cite[Section 7.3]{DZ2010}, which asserts that as $N \to \infty$, conditioned on the event $\{ Y_1  + \ldots + Y_N \leq t N \}$, the $k$-dimensional random vector $(Y_1,\ldots,Y_k)$ converges in distribution to the $k$-dimensional product of the measure with density
\begin{align*}
\psi_{\alpha_t}(y)  := e^{ \alpha_t y} \psi(y) / Z(\alpha_t),
\end{align*} 
where $\alpha_t < 0 $ is chosen so that $\int_{-\infty}^\infty \gamma_{\alpha_t}(y) y \mathrm{d} y = t$. 

The quantitative Gibbs conditional principle, which we state shortly, gives a  statement of this result in terms of total variation distances. Namely, define the constants 
\begin{align*}
\xi_k := \frac{1}{2} \mathbb{E} \left[ \Bigg| 1 - \frac{ Z_1 + \ldots + Z_k - k \mathbb{E}[ Z_1 ] }{ \sqrt{ k \mathrm{Var}[Z_1]} } \Bigg| \right].
\end{align*}
By the central limit theorem, as $k \to \infty$,
\[\xi_{k} \to \xi := \frac{1}{2} \int_{ - \infty}^\infty \frac{e^{ - \frac{1}{2} s^2 } }{ \sqrt{2 \pi }}  \big| 1 - s^2 \big| \, \dint s = \sqrt{ \frac{2}{ \pi e} },
\]
where the final equality above follows from noting  $\frac{ \dint}{ \dint u } \left( u e^{ - u^2/2} \right) = (1 - u^2 ) e^{ - u^2/2}$. Now for $\theta \in (0,1)$ define
\begin{align*}
Q(\theta) := \int_{ -\infty}^\infty \Big| 1 - \sqrt{ 1 - \theta } e^{ \theta \zeta^2/2} \Big| \frac{1}{ \sqrt{ 2 \pi }} \exp \left( - \frac{1}{2} \zeta^2 \right) \dint \zeta. 
\end{align*}
It is easily verified that $Q'(0) = \xi$. 

We now state our quantitative version of the Gibbs conditioning principle, with a minor restatement of the form given in \cite{DF2}.

\begin{lemma}[Theorem 1.6 of \cite{DF2}] \label{lem:qgibbs}
Let $\psi_{ \geq t}^{N \to k}$ be the marginal density of $(Y_1,\ldots,Y_k)$ conditioned on the event $\{ Y_1 + \ldots + Y_N \geq  tN \}$. Let $\alpha_t$ be the such that $\frac{ \partial}{ \partial \alpha}\log  Z(\alpha)\Big|_{\alpha = \alpha_t} = t$. Then
\begin{align*}
\int_{\mathbb{R}^k} \Big| \psi_{\geq t}^{N \to k }(y) - \psi_{\alpha_t}^{ \otimes k}(y) \Big| \dint y= 
\begin{cases}
\left( 1 + \varepsilon^{(1)}_{k,N} \right) \xi_{k,\alpha_t}\frac{k}{N} \qquad &: \text{$k$ fixed, $N \to \infty$},\\
\left( 1 + \varepsilon^{(2)}_{k,N}  \right) \xi \frac{k}{N} \qquad &:\text{$k = o(N), k,N \to \infty$},\\
\left( 1 + \varepsilon^{(3)}_{k,N}  \right) Q(\theta) \qquad &:\text{$k \sim \theta N, k,N \to \infty$},
\end{cases}
\end{align*}
where there is a universal constant $C \in (0,\infty)$ such that setting $C' := C   \frac{ \mathbb{E} [ | X - \mathbb{E}[X]|^4] }{ \mathrm{Var}[X]^{3/2} }$ we have 
\begin{align*}
|\varepsilon^{(1)}_{k,N}| \leq C' \sqrt{ \frac{k}{N} }, \qquad |\varepsilon^{(2)}_{k,N}|, |\varepsilon^{(3)}_{k,N}| \leq C' \left( \frac{1}{ \sqrt{k}}  + \frac{1}{ \sqrt{N-k}} \right).
\end{align*}
\end{lemma}

We now have all the tools at hand to prove Theorem \ref{thm:main}, which we do over the next two sections.

% % % % % % % % % % % % % % % % % % % %
% % % % % % % % % % % % % % % % % % % %
% % % % % % % % % % % % % % % % % % % %
% % % % % % % % % % % % % % % % % % % %
\section{Proof of Theorem \ref{thm:main}: the $t > t_{\mathrm{crit}}$ case} \label{sec:main2}
% % % % % % % % % % % % % % % % % % % %
% % % % % % % % % % % % % % % % % % % %
% % % % % % % % % % % % % % % % % % % %
% % % % % % % % % % % % % % % % % % % %

Recall that according to the first point in Theorem \ref{thm:main}, for all $t_{\mathrm{crit}} < t < t_{\mathrm{sup}}$ there are constants $C,c \in (0,\infty)$ depending on $\phi$ and $t$ such that for all $k,N\in\N$ with $k\leq N$, we have

\begin{align} \label{eq:TVa}
\int_{\mathbb{R}^k}  \Big| \mu^{N \to k}_{\phi,t}(s) - \gamma_{\phi,\mathrm{uni}}^{\otimes k }(s) \Big| \, \dint s \leq C e^{ - cN}.
\end{align}
As mentioned in the introduction, we in fact prove the following stronger statement, of which the bound \eqref{eq:TVa} is a consequence.

\begin{thm} \label{thm:uni full}
If $t_{\mathrm{crit}} < t < t_{\mathrm{sup}}$, 
\begin{align*}
\int_{\mathbb{R}^k}  \Big| \mu^{N \to k}_{\phi,t}(s) - \gamma_{\phi,\mathrm{uni}}^{\otimes k }(s) \Big| \, \dint s =  (2 +  \varepsilon_{N,k} ) \frac{1}{\sqrt{ 2 \pi \sigma^2_t N } }  e^{ - I(t) N },
\end{align*}
where the rate function $I:[0,t_{\mathrm{sup}}) \to [0,\infty]$ is given by $I(t) := t \alpha_t - \log Z(\alpha_t)$, $\sigma^2_t := \frac{ \partial^2}{ \partial \alpha^2} \log Z(\alpha) |_{\alpha = \alpha_t}$ and there are constants $C=C(\phi,t),c = c(\phi,t) \in(0,\infty)$ such that $| \varepsilon_{N,k} | \leq  C \left( \frac{1}{ \sqrt{N}} +  \frac{k}{N} + e^{ - c \sqrt{k}} \right)$ for all $k,N$.
\end{thm}

In the remainder of this section we prove Theorem \ref{thm:uni full}. We begin by consider the probability density $Q_{\phi,t}^{N \to k}$ on $\mathbb{R}^k$ given by the conditional law of the first $k$ coordinates of a random vector uniformly distributed on $D_{\phi}^N \setminus B_{\phi,t}^N$. Namely,
\begin{align} \label{eq:intrep1}
Q_{\phi,t}^{N \to k}( s) := \frac{1}{ |D_\phi|^N - | B_{\phi,t}^N | } \int_{ \mathbb{R}^{N -k } } \ind_{ \Big\{ (s, \zeta) \in D_\phi^N \setminus B_{\phi,t}^N \Big\}}  \dint \zeta ,
\end{align} 
where for $s = (s_1,\ldots,s_k) \in \mathbb{R}^k$ and $\zeta = (\zeta_1,\ldots,\zeta_{N-k}) \in \mathbb{R}^{N-k}$, we write $(s,\zeta) := (s_1,\ldots,s_k,\zeta_1,\ldots,\zeta_{N-k}) \in \mathbb{R}^N$. Moreover, we note that by definition $\mu_{\phi,t}^{N \to k}$ may also be written as an integral over $\mathbb{R}^{N -k}$:
\begin{align} \label{eq:intrep2}
\mu_{\phi,t}^{N \to k}( s) := \frac{1}{ | B_{\phi,t}^N | } \int_{ \mathbb{R}^{N -k } } \ind_{ \Big\{ (s, \zeta) \in B_{\phi,t}^N \Big\}}  \dint \zeta .
\end{align}

 We now work to express the total variation distance between $\mu_{\phi,t}^{N \to k}(s)$ and $\gamma_{\phi,\mathrm{uni}}^{ \otimes k}$ in terms of $Q_{\phi,t}^{N \to k}$. Indeed, by definition we have
\begin{align} \label{eq:apple}
\int_{ \mathbb{R}^k } \Big| \mu_{\phi,t}^{N \to k }(s) - \gamma_{ \phi, \mathrm{uni} }^{ \otimes k} (s) \Big| \dint s := \int_{ \mathbb{R}^k } \Bigg| \int_{ \mathbb{R}^{N-k} }  \left\{  \frac{ \ind_{ \{ (s,\zeta) \in B_{\phi,t}^N \}} }{ |B_{\phi,t}^N | }  - \frac{ \ind_{ \{ (s,\zeta) \in D_\phi^N \}} }{ |D_\phi^N | }    \right\}  \dint  \zeta \Bigg| \dint s.
\end{align}
Now, for each $s \in \mathbb{R}^k$, since $B_{\phi,t}^N \subseteq D_\phi^N$ we have 
\begin{align} 
 &\int_{ \mathbb{R}^{N-k} } \left\{  \frac{ \ind_{ \{ (s,\zeta) \in B_{\phi,t}^N \}} }{ |B_{\phi,t}^N | }  - \frac{ \ind_{ \{ (s,\zeta) \in D_\phi^N \}} }{ |D_\phi^N | }    \right\} \dint  \zeta \nonumber  \\
&=  \left( \frac{1}{ |B_{\phi,t}^N |} - \frac{1}{ |D_\phi|^N } \right) \int_{ \mathbb{R}^{N-k} }  \ind_{ \Big\{ (s,\zeta) \in B_{\phi,t}^N \Big\}} \dint \zeta - \frac{1}{ |D_\phi|^N } 
\int_{\mathbb{R}^{N-k} }  \ind_{ \Big\{ (s, \zeta) \in D_\phi^N \setminus B_{\phi,t}^N \Big\}} \dint \zeta  \nonumber  \\
&= \left( 1 - \frac{|B_{\phi,t}^N|}{ |D_\phi|^N } \right) \left( \mu_{\phi,t}^{N \to k }( s) - Q_{\phi,t}^{N \to k }(s) \right), \label{eq:orange}
\end{align}
where we used \eqref{eq:intrep1} and \eqref{eq:intrep2} to obtain the final equality above. Plugging \eqref{eq:orange} into \eqref{eq:apple}, we obtain 
\begin{align} \label{eq:mango}
 \int_{ \mathbb{R}^k } \Big| \mu_{\phi,t}^{N \to k }(s) - \gamma_{ \phi, \mathrm{uni} }^{ \otimes k} (s) \Big| \dint s  =\left( 1 - \frac{|B_{\phi,t}^N|}{ |D_\phi|^N } \right) \int_{ \mathbb{R}^k } \Big| Q_{\phi,t}^{N \to k } (s) - \gamma_{ \phi, \mathrm{uni} }^{ \otimes k } (s) \Big| \dint s.
\end{align}

The following lemma is the main part of the proof, giving a fine estimate of the integral occuring in \eqref{eq:mango}.

\begin{lemma} \label{lem:mango}
Fix $\theta \in (0,1)$. There are constants $C = C(\theta),c = c(\theta) \in (0,\infty)$ such that for all integers $k,N$ such that $k \leq \theta N$ we have
\begin{align*}
 \int_{ \mathbb{R}^k } \Big| Q_{\phi,t}^{N \to k } (s) - \gamma_{ \phi, \mathrm{uni} }^{ \otimes k } (s) \Big| \dint s = 2 - \varepsilon_{N,k},
\end{align*}
where $0 \leq \varepsilon_{N,k}  \leq C \left( \frac{k}{N} + e^{ - c k }  \right)$.
\end{lemma}

\begin{proof}

Note that both $Q_{\phi,t}^{N \to k}$ and $\gamma^{\otimes k}_{\phi,(\alpha_t)}$ are supported on $D_\phi^k$. Moreover, consider now that the density $Q_{\phi,t}^{N \to k}$ may be written $Q_{\phi,t}^{N \to k}( s) =  f \left( \Phi(s) \right)$, where $\Phi:D_\phi^k \to [0,\infty)$ is given by $\Phi(s_1,\ldots,s_k) := (\phi(s_1),\ldots,\phi(s_k))$ and $f:[0,\infty)^k \to [0,\infty)$ is given by 
\begin{align*}
f(y_1,\ldots,y_k) := \frac{1}{ |D_\phi|^N - |B_{\phi,t}^N| } \int_{ \mathbb{R}^{N-k}  }  \ind_{ \left\{ \phi(\zeta_1) + \ldots + \phi(\zeta_{N-k} > t N - \sum_{ i = 1}^k y_i \right\} } \dint \zeta_1 \ldots \dint \zeta_{N-k}.
\end{align*}
Similarly, we may also write $\gamma_{\phi,\mathrm{uni}}^{ \otimes k}(s) := g \left( \Phi(s) \right)$, where $g(y_1,\ldots,y_k) := \frac{1}{|D_\phi^N|}$ (i.e. a multiple of the constant function). In particular, we are in the setting of Lemma \ref{lem:pushforwards} with $E = D_\phi^k$, $F = [0,\infty)^k$, with $\mu$ equal to the $k$-dimensional Lebesgue measure on $E$. It follows that 
\begin{align} \label{eq:push rep}
 \int_{ \mathbb{R}^k } \Big| Q_{\phi,t}^{N \to k } (s) - \gamma_{ \phi, \mathrm{uni} }^{ \otimes k } (s) \Big| \dint s  =  \int_{ [0,\infty)^k } \Big| \Phi^\# Q_{\phi,t}^{N \to k } (y ) -  \Phi^\# \gamma_{ \phi, \mathrm{uni} }^{ \otimes k } (y) \Big| \dint y  ,
\end{align}
where $\Phi^\# Q_{\phi,t}^{N \to k }$ and $ \Phi^\# \gamma_{ \phi, \mathrm{uni} }^{ \otimes k }$ denote the respective densities on $[0,\infty)^k$ of the random vectors\\ $(\phi(X_1),\ldots,\phi(X_k))$  and $(\phi(Y_1),\ldots,\phi(Y_k))$ where $(X_1,\ldots,X_k)$ is distributed according to density $Q_{\phi,t}^{N \to k }$ and $(Y_1,\ldots,Y_k)$ is distributed according to density $\gamma_{\phi,\mathrm{uni}}^{ \otimes k}$.

Now note that $\Phi^{\#} Q_{\phi,t}^{N \to k}$ is precisely the conditional density of $\left( \phi(X_1),\ldots,\phi(X_k)\right)$ conditioned on the event $\left \{ \phi(X_1) + \ldots + \phi(X_N) > t N \right\}$ where $X_1,\ldots,X_N$ are  independent and uniformly distributed on $D_\phi$. Equivalently, by \eqref{eq:zero push}, $\Phi^{\#} Q_{\phi,t}^{N \to k}$ is the conditional density of $(Y_1,\ldots,Y_k)$ conditioned on the event $\{ Y_1 + \ldots + Y_N > t N \}$, where $Y_1,\ldots,Y_N$ are independent and identically distributed with density $\psi_0(y) := \psi(y)/|D_\phi|$. In particular, applying Lemma \ref{lem:qgibbs} to the variables $Y_1,\ldots,Y_k$, and extracting a rather rough bound from Lemma \ref{lem:qgibbs}, we see that there is a constant $C = C(\phi,t)\in (0,\infty)$ such that for all $k,N\in\N$ with $k\leq N$
\begin{align} \label{eq:rough bound}
\int_{\mathbb{R}^k} \Big| \Phi^\# Q_{\phi,t}^{ N \to k} (y) - \psi_{\alpha_t} ^{ \otimes k } (y) \Big| \dint y \leq C \frac{k}{N}
\end{align} 
where $\alpha_t$ is chosen so that $t =\frac{ \partial}{ \partial \alpha} \log Z(\alpha)|_{\alpha=\alpha_t}$. (We remark that the extra precision granted by Lemma \ref{lem:qgibbs} is used more finely in the next section in our study of the case $t \leq t_{\mathrm{crit}}$.)

In particular, using \eqref{eq:push rep}, \eqref{eq:rough bound} and the triangle inequality we have 
\begin{align} \label{eq:q guy}
 \int_{ \mathbb{R}^k } \Big| Q_{\phi,t}^{N \to k } (s) - \gamma_{ \phi, \mathrm{uni} }^{ \otimes k } (s) \Big| \dint s =  \int_{ [0,\infty)^k } \Big| \Phi^{ \#} \gamma_{\phi,\mathrm{uni}}^{ \otimes k} (y) - \psi_{\alpha_t}^{ \otimes k }(y) \Big| \dint y + \Delta_{k,N},
\end{align}
where $|\Delta_{k,N} | \leq Ck/N$. 

We now note that $\Phi^{ \# } \gamma_{\phi,\mathrm{uni}}^{\otimes k}  = \psi^{ \otimes k}_0$. In particular, 
\[ \int_{ [0,\infty)^k } \Big| \Phi^{ \#} \gamma_{\phi,\mathrm{uni}}^{ \otimes k} (y) - \psi_{\alpha_t}^{ \otimes k }(y) \Big| \dint y  = 2 \sup_{ A \subseteq [0,\infty)^k } \Big| \int_A \psi_{\alpha_t}^{\otimes k}(y) \dint y - \int_A \psi^{ \otimes k }_0 (y) \dint y\Big|\]
 is the total variation between the distributions of $(Y_1,\ldots,Y_k)$ and $(Y_1',\ldots,Y_k')$, where the $Y_i$ are i.i.d. with density $\psi$, and in particular have mean $t_{\mathrm{crit}}$, and $Y_i'$ are i.i.d. with density $\psi_{\alpha_t}$, and in particular have mean $t$. We are going to show that this total variation is nearly equal to $2$ when $k$ is large. Indeed, if we set
\[u := \frac{ t_{\mathrm{crit}} + t }{ 2},
\]
then for large $k$ it is likely that $\left\{ \frac{ Y_1 + \ldots + Y_k}{ k } \leq u \right\}$ but unlikely that $\left\{  \frac{ Y'_1 + \ldots + Y'_k}{ k } \leq u \right\}$. More explicitly, setting $A_u := \Big\{ (y_1,\ldots,y_k) \in [0,\infty)^k : y_1 + y_2 + \ldots + y_k > u k \Big\}$ and extracting a rather rough bound from the quantitative Cram\'er theorem, Lemma \ref{lem:qcramer}, we see that there exist constants $C=C(\phi,t),c=c(\phi,t) \in (0,\infty)$ such that 
\begin{align*}
\int_{A_u} \psi^{\otimes k}(y) \dint y \leq C e^{ - c k} \qquad \text{and} \qquad \int_{A_u} \psi_{\alpha}^{ \otimes k} (y) \dint y \geq 2 - C e^{ - ck}.
\end{align*}
In particular, two previous two estimates imply that
\begin{align} \label{eq:q x}
2 \geq \int_{ [0,\infty)^k } \Big| \Phi^{ \#} \gamma_{\phi,\mathrm{uni}}^{ \otimes k} (y) - \psi_{\alpha_t}^{ \otimes k }(y) \Big| \dint y \geq 2 - 2C e^{ - ck}.
\end{align}
Combining \eqref{eq:q guy} with \eqref{eq:q x}, we obain the result.

\end{proof}

We are now ready to prove Theorem \ref{thm:uni full}.

\begin{proof}[Proof of Theorem \ref{thm:uni full}]
Consider the large-$N$ asymptotics of the right hand side of \eqref{eq:mango}.

On the one hand, we may write
\begin{align*}
\left( 1 - \frac{|B_{\phi,t}^N|}{ |D_\phi|^N } \right) = \mathbb{P} \left[ \phi(V_1) + \ldots + \phi(V_N) > t N \right],
\end{align*}
where $V_i$ are independent random variables distributed according to $\gamma_{\phi,\mathrm{uni}}$, the uniform density on $D_\phi$. By the quantitative version of Cramer's theorem, Lemma \ref{lem:qcramer}, we have 
\begin{align*}
\left( 1 - \frac{|B_{\phi,t}^N|}{ |D_\phi|^N } \right)  = (1 + \varepsilon_N)  \frac{1}{ \sqrt{2 \pi \sigma^2_t N } } \exp ( - I(t) N ),
\end{align*} 
where $|\varepsilon_N | < C/\sqrt{N}$ for a constant $C = C(\phi,t) \in(0,\infty)$ not depending on $N$.  

On the other hand, by Lemma \ref{lem:mango} we have
 \begin{align*}
 \int_{ \mathbb{R}^k } \Big| Q_{\phi,t}^{N \to k } (s) - \gamma_{ \phi, \mathrm{uni} }^{ \otimes k } (s) \Big| \dint s = 2 - \varepsilon_{N,k},
\end{align*}
where for a different constant $C' = C'(\phi)$ we have $|\varepsilon_{N,k} | \leq C' \left( \frac{k}{N} + e^{ - c \sqrt{k}}  \right).$

In particular, by \eqref{eq:mango}, 
\begin{align*}
 \int_{ \mathbb{R}^k } \Big| \mu_{\phi,t}^{N \to k }(s) - \gamma_{ \phi, \mathrm{uni} }^{ \otimes k} (s) \Big| \dint s  =  (1 + \varepsilon_N) ( 2 - \varepsilon_{N,k} )  \frac{1}{ \sqrt{2 \pi \sigma^2_t N } } \exp ( - I(t) N ), 
\end{align*}
where $|\varepsilon_{N,k} | \leq C' \left( \frac{k}{N} + e^{ - c \sqrt{k}}  \right)$ and $|\varepsilon_N | < C/\sqrt{N}$. Set $\rho_{N,k}$ to be the solution to 
\begin{align*}
1 + \rho_{N,k} =  (1 + \varepsilon_N) ( 1 - \frac{1}{2} \varepsilon_{N,k} )  
\end{align*}
Then plainly there are constants $c,C\in(0,\infty)$ such that $|\rho_{N,k}| \leq C \left( \frac{1}{\sqrt{N}} + \frac{k}{N} + e^{ - ck} \right)$, completing the proof of Theorem \ref{thm:uni full}.

\end{proof}
% % % % % % % % % % % % % % % % % % %
% % % % % % % % % % % % % % % % % % % %
% % % % % % % % % % % % % % % % % % % %
% % % % % % % % % % % % % % % % % % % %
\section{Proof of Theorem \ref{thm:main}: the $t \leq t_{\mathrm{crit}}$ case} \label{sec:main}
% % % % % % % % % % % % % % % % % % % %
% % % % % % % % % % % % % % % % % % % %
% % % % % % % % % % % % % % % % % % % %
% % % % % % % % % % % % % % % % % % % %

\subsection{A full statement and overview}

We now turn to proving Theorem \ref{thm:main} in the case where $t \leq t_{\mathrm{crit}}$. We recall from Section \ref{sec:results} that $\alpha_t \leq 0$ is a parameter chosen so that if $\gamma_{\phi,\alpha}(s)$ is the tilted density
\begin{align*}
\gamma_{\phi,\alpha}(s) := e^{ \alpha \phi(s) } \ind_{D_\phi(s) } / Z(\alpha),
\end{align*}
then $\int_{-\infty}^\infty \phi(s) \gamma_{\phi,\alpha}(s) \mathrm{d} s$. 

As in the $t > t_{\mathrm{crit}}$, we actually prove the following sharper result, giving a fine estimate of the total variation which implies \eqref{eq:TV2}. 

\begin{thm} \label{thm:alpha full}
If $t \leq t_{\mathrm{crit}}$, then with $\alpha_t$ as in Lemma \ref{lem:parameter} we have 
\begin{align*}
\int_{\mathbb{R}^k}  \Big| \mu^{N \to k}_{\phi,t}(s) - \gamma_{\phi,\alpha_t}^{\otimes k }(s) \Big| \, \dint s =
\begin{cases}
\left( 1 + \varepsilon^{(1)}_{k,N} \right) \xi_{k,\alpha_t}\frac{k}{N} \qquad &: \text{$k$ fixed, $N \to \infty$},\\
\left( 1 + \varepsilon^{(2)}_{k,N}  \right) \xi \frac{k}{N} \qquad &:\text{$k = o(N), k,N \to \infty$},\\
\left( 1 + \varepsilon^{(3)}_{k,N}  \right) Q(\theta) \qquad &:\text{$k \sim \theta N, k,N \to \infty$},
\end{cases}
\end{align*}
and for a constant $C = C(\phi,t) \in(0,\infty)$ we have
\begin{align*}
|\varepsilon^{(1)}_{k,N}| \leq C \sqrt{ \frac{k}{N} }, \qquad |\varepsilon^{(2)}_{k,N}|,  |\varepsilon^{(2)}_{k,N}|  \leq C \left(  \frac{1}{ \sqrt{k} } + \frac{1}{ \sqrt{N-k}} \right).
\end{align*}
\end{thm}

Theorem \ref{thm:alpha full} is proved in the remainder of this section. The proof is divided into three steps:

\begin{itemize}
\item Step 1. We assume that the Lebesgue measure of $D_\phi$ is finite, and under this assumption Lemma \ref{lem:finite push} below estimates the total variation between the pushforwards $\Phi^{ \#} \mu_{\phi,t}^{N \to k}$ and $\Phi^{ \#} \gamma_{\phi,\alpha_t}^{\otimes k}$.
\item Step 2. We will continue to assume $|D_\phi| < \infty$, and use our work in Section \ref{sec:sufficiency} to show that we may `pullback' the result obtained in Lemma \ref{lem:finite push} to estimate the total variation between the densities $\mu_{\phi,t}^{N \to k}$ and $\gamma_{\phi,t}^{ \otimes k}$ .
\item Step 3. Finally, we will show that the assumption $|D_\phi| < \infty$ may be lifted, completing the proof of Theorem \ref{thm:alpha full}. This part is based on a truncation argument, in which the domain $D_\phi$ of infinite Lebesgue measure is approximated by sets with finite Lebesgue measure.
\end{itemize}

The next three sections correspond to the three steps outlined above.

\subsection{Step 1}

The following lemma may be regarded as a pushforward version of Theorem \ref{thm:alpha full}

\begin{lemma} \label{lem:finite push}
Suppose $|D_\phi| < \infty$ and $t \leq t_{\mathrm{crit}}$. Then, with $\alpha_t \leq 0$ as in Definition \ref{lem:parameter}, we have
\begin{align*}
\int_{ \mathbb{R}^k } \Bigg| \Phi^{ \#} \mu_{\phi,t}^{N \to k} (y) - \psi_{\alpha_t}^{\otimes k} (y)  \Bigg|  \dint y =
\begin{cases}
\left( 1 + \varepsilon^{(1)}_{k,N} \right) \xi_{k,\alpha_t}\frac{k}{N} \qquad &: \text{$k$ fixed, $N \to \infty$},\\
\left( 1 + \varepsilon^{(2)}_{k,N}  \right) \xi \frac{k}{N} \qquad &:\text{$k = o(N), k,N \to \infty$},\\
\left( 1 + \varepsilon^{(3)}_{k,N}  \right) Q(\theta) \qquad &:\text{$k \sim \theta N, k,N \to \infty$},
\end{cases}
\end{align*}
where there is a universal constant $C$ with the following property. Namely, if $X$ is a random variable distributed according to $\psi_{\alpha_t}$, then setting $C' := C \frac{ \mathbb{E}[ (X- \mathbb{E}[X])^4 ]}{ \mathrm{Var}[X]^2}$ we have
\begin{align*}
|\varepsilon^{(1)}_{k,N}| \leq C' \sqrt{ \frac{k}{N} }, \qquad |\varepsilon^{(2)}_{k,N}|, |\varepsilon^{(3)}_{k,N}|  \leq C' \left( \frac{1}{\sqrt{k}} + \frac{1}{\sqrt{N-k}} \right).
\end{align*}
\end{lemma}

\begin{proof}
Since $|D_\phi|<\infty$, the uniform density $\gamma_{\phi,\mathrm{uni}}(s) := \frac{\ind_{D_\phi}(s) }{ |D_\phi |}$ on $D_\phi$ exists. Suppose now $X_1,\ldots,X_N$ are independent random variables distributed according to the uniform density $\gamma_{\phi,\mathrm{uni}}$ on $D_\phi$, and consider the transformed variables $\phi(Y_1),\ldots,\phi(Y_N)$, which are distributed according to $\psi_0$, where $\psi_0$ is given in \eqref{eq:zero push}. Noting that by definition $\mu_{\phi,t}^{N \to k}$ is the marginal density of $(X_1,\ldots,X_k)$ conditioned on the event $\{ \phi(X_1) + \ldots + \phi(X_N) \leq t N\}$, it follows that the pushforward $\Phi^{ \#} \mu_{\phi,t}^{N \to k}$ is the marginal density of $(Y_1,\ldots,Y_k)$ conditioned on the event $\{ Y_1 + \ldots + Y_N \leq t N \}$. In other words, we are in the setting of Lemma \ref{lem:qgibbs}, from which the result follows immediately.
\end{proof}

\subsection{Step 2}
The next lemma `pulls back' the previous result, replacing the estimate of the total variation between $\Phi^{ \#} \mu_{\phi,t}^{N \to k}$ and $\psi_{\alpha_t}^{ \otimes k}$ with one between $\mu_{\phi,t}^{N \to k}$ and $\gamma_{\alpha_t}^{\otimes k}$. This next result amounts to preicsely the statement of Theorem \ref{thm:alpha full} in the case where the Lebesgue measure of $D_\phi$ is finite. 
\begin{lemma} \label{lem:finite pull}
Suppose $|D_\phi| < \infty$ and $t \leq t_{\mathrm{crit}}$. Then, with $\alpha_t \leq 0$ as in Definition \ref{lem:parameter}, we have
\begin{align*}
\int_{ \mathbb{R}^k } \Bigg|  \mu_{\phi,t}^{N \to k} (s) - \gamma_{\alpha_t}^{\otimes k} (s)  \Bigg|  \dint s =
\begin{cases}
\left( 1 + \varepsilon^{(1)}_{k,N} \right) \xi_{k,\alpha_t}\frac{k}{N} \qquad &: \text{$k$ fixed, $N \to \infty$},\\
\left( 1 + \varepsilon^{(2)}_{k,N}  \right) \xi \frac{k}{N} \qquad &:\text{$k = o(N), k,N \to \infty$},\\
\left( 1 + \varepsilon^{(3)}_{k,N}  \right) Q(\theta) \qquad &:\text{$k \sim \theta N, k,N \to \infty$},
\end{cases}
\end{align*}
where there is a universal constant $C$ with the following prooperty. Namely, if $X$ is a random variable distributed according to $\psi_{\alpha_t}$, then setting $C' := C \frac{ \mathbb{E}[ (X- \mathbb{E}[X])^4 ]}{ \mathrm{Var}[X]^2}$ we have
\begin{align*}
|\varepsilon^{(1)}_{k,N}| \leq C' \sqrt{ \frac{k}{N} }, \qquad |\varepsilon^{(2)}_{k,N}|, |\varepsilon^{(3)}_{k,N}|  \leq C' \left( \frac{1}{\sqrt{k}} + \frac{1}{\sqrt{N-k}} \right).
\end{align*}
\end{lemma}

\begin{proof}
We would like to use Lemma \ref{lem:pushforwards}, with $\Phi:\mathbb{R}^k \to [0,\infty]^k$ defined as in Section \ref{sec:potential}: i.e. $\Phi(s_1,\ldots,s_k) = \left( \phi(s_1),\dots,\phi(s_k) \right)$. To this end note that we may write  $\mu_{\phi,t}^{N \to k }( s) = f \left( \Phi(s) \right)$, where
\begin{align*}
f( y_1,\ldots, y_k ) := \frac{1}{ |B_{\phi,t}^N | } \int_{ \mathbb{R}^{N -k } } \ind_{\left\{ \phi(\zeta_1) + \ldots + \phi(\zeta_{N-k} ) \leq t N - \sum_{ i = 1}^k y_i \right\}} \dint \zeta_1 \ldots \dint \zeta_{N-k} .
\end{align*}
Moreover, clearly $\gamma_{\phi,\alpha_t}^{ \otimes k}(s) = g\left( \Phi(s) \right)$, where $g(y_1,\ldots,y_k) = \frac{ e^{ \alpha_t( y_1 + \ldots + y_k) } }{ Z(\alpha)^k }$. In particular, we are in the setting of Lemma \ref{lem:pushforwards}, so that 
\begin{align*}
\int_{ \mathbb{R}^k } \Bigg|  \mu_{\phi,t}^{N \to k} (s) - \gamma_{\phi,\alpha_t}^{\otimes k} (s)  \Bigg|  \dint s = \int_{ \mathbb{R}^k } \Bigg|  \Phi^\# \mu_{\phi,t}^{N \to k} (y) - \Phi^\# \gamma_{\phi,\alpha_t}^{\otimes k} (y)  \Bigg|  \dint y.
\end{align*}
The result follows by noting that $\Phi^\# \gamma_{\alpha_t}^{\otimes k} = \psi_{\alpha_t}^{\otimes k}$, and using Lemma \ref{lem:finite push}.
\end{proof}

\subsection{Step 3}

In the previous section, we proved Lemma \ref{lem:finite pull}, which is precisely the statement that Theorem \ref{thm:alpha full} holds whenever $D_\phi$ has finite Lebesgue measure. Here we will show that this assumption may be lifted, thereby completing the proof of Theorem \ref{thm:alpha full}.

\begin{proof}[Proof of Theorem \ref{thm:alpha full}]
We approximate an arbitrary potential $\phi$ by a potential $\phi_L$ whose domain has finite Lebesgue measure. Namely, for $L > 0$ let $\phi_L(s) := \phi(s) + \infty \ind \{ \phi(s) > L \}$, so that in particular $\phi$ agrees with $\phi_L$ on the set $\{ s\in\R \,:\, \phi(s) \leq L \}$. 

Define the truncated partition function
\begin{align*}
Z_L (\alpha) := \int_{ D_{ \phi_L}  } e^{\alpha \phi_L(s) } \dint s = \int_{ \phi^{ - 1} \left( [0,L] \right) } e^{ \alpha \phi(s) } \dint s = \int_0^L e^{ \alpha y} \psi(y) \dint y,
\end{align*}
where the final equality above is a truncated analogue of \eqref{eq:new PF}, and follows from an application of \eqref{eq:change}. For $t \in \mathbb{R}$ let $\alpha_{t,L}$ denote the solution to $\frac{ \partial}{ \partial \alpha } Z_L(\alpha)|_{\alpha = \alpha_{t,L}} = t$.

Now by the triangle inequality, for any $L > 0$ we have the inequalities
\begin{align*}
||  \mu_{\phi,t}^{N \to k}- \gamma_{\phi,\alpha_t}^{\otimes k} || \leq  ||  \mu_{\phi_L,t}^{N \to k}- \gamma_{\phi_L,\alpha_{t,L}}^{\otimes k} || + || \mu_{\phi_L,t}^{N \to k} -  \mu_{\phi,t}^{N \to k}  || +   || \gamma_{\phi_L,\alpha_{t,L}}^{\otimes k}  -  \gamma_{\phi,\alpha_t }^{\otimes k} || 
\end{align*} 
and 
\begin{align*}
||  \mu_{\phi,t}^{N \to k}- \gamma_{\phi,\alpha_t}^{\otimes k} || \leq  ||  \mu_{\phi_L,t}^{N \to k}- \gamma_{\phi_L,\alpha_{t,L}}^{\otimes k} || - || \mu_{\phi_L,t}^{N \to k} -  \mu_{\phi,t}^{N \to k}  || - || \gamma_{\phi_L,\alpha_{t,L}}^{\otimes k}  -  \gamma_{\phi,\alpha_t }^{\otimes k} || .
\end{align*} 
In particular, we may write 
\begin{align} \label{eq:tarantula}
\int_{ \mathbb{R}^k } \Bigg|  \mu_{\phi,t}^{N \to k} (s) - \gamma_{\phi,\alpha_t}^{\otimes k} (s)  \Bigg|  \dint s = \int_{ \mathbb{R}^k } \Bigg|  \mu_{\phi_L,t}^{N \to k} (s) - \gamma_{\phi_L,\alpha_{t,L}}^{\otimes k} (s)  \Bigg|  \dint s  + \Delta_L ,
\end{align} 
where
\begin{align} \label{eq:pastel}
|\Delta_L| \leq \int_{ \mathbb{R}^k } |  \gamma_{\phi_L,\alpha_{t,L}}^{\otimes k} (s)  -  \gamma_{\phi,\alpha_t }^{\otimes k} (s)  | \dint s + \int_{ \mathbb{R}^k } \Big|  \mu_{\phi_L,t}^{N \to k} (s) -  \mu_{\phi,t}^{N \to k} (s) \Big| \dint s.
\end{align}
We now show that for all $\rho > 0$, there exists an $L_0$ such that whenever $L \geq L_0$, $|\Delta_L| < \rho$. On the one hand, we note that the truncated Orlicz ball $B_{\phi_L,t}^N$ is a subset of $B_{\phi,t}^N$, both of which have finite Lebesgue measure. Moreover, as $L \to \infty$, $B_{\phi,t}^N \setminus B_{\phi_L,t}^N \to \varnothing$, so that by the continuity of Lebesgue measure we have $| B_{\phi,t}^N \setminus B_{\phi_L,t}^N | \to 0$. In particular, by Lemma \ref{lem:TV projection 2} there is an $L_1$ such that whenever $L \geq L_1$, we have
\begin{align} \label{eq:vespers1}
\int_{ \mathbb{R}^k } \Big|  \mu_{\phi_L,t}^{N \to k} (s) -  \mu_{\phi,t}^{N \to k} (s) \Big| \dint s \leq \rho /2.
\end{align} 
On the other hand, using Lemma \ref{lem:TVsum} we have
\begin{align} \label{eq:vigil}
\int_{ \mathbb{R}^k } |  \gamma_{\phi_L,\alpha_{t,L}}^{\otimes k} (s)  -  \gamma_{\phi,\alpha_t }^{\otimes k} (s)  | \dint s \leq k  \int_{ \mathbb{R} } |  \gamma_{\phi_L,\alpha_{t,L}} (s)  -  \gamma_{\phi,\alpha_t } (s)  | \dint s 
\end{align}
Now consider that $\gamma_{\phi,\alpha_t}(s) = \frac{ e^{ \alpha_t \phi(s) }}{ Z(\alpha_t)}$ and $\gamma_{\phi_L, \alpha_{t,L}} (s) := \frac{ e^{ \alpha_{t,L} \phi(s)} }{ Z_L(\alpha_t)}$ are both densities that may be written as a function of $\phi$. In particular, we are in the setting of Lemma \ref{lem:pushforwards} with $\Phi = \phi:\mathbb{R} \to [0,\infty]$, so that
\begin{align} \label{eq:monk}
  \int_{ \mathbb{R} } |  \gamma_{\phi_L,\alpha_{t,L}} (s)  -  \gamma_{\phi, \alpha_t } (s)  | \dint s  =   \int_0^\infty | \phi^{\#}  \gamma_{\phi_L, \alpha_{t,L}} (y)  -  \phi^{\#} \gamma_{\phi,\alpha_t } (y)  | \dint y.
\end{align}
We now note that in the notation of Section \ref{sec:tech}, we have $\phi^{\#}  \gamma_{\phi_L, \alpha_{t,L}} = \left( \phi^{\#}  \gamma_{\phi, \alpha_t } \right)_{\langle L \rangle}$, so that making the relevant substitution in \eqref{eq:monk} we have
\begin{align*}
  \int_{ \mathbb{R} } |  \gamma_{\phi_L,\alpha_{t,L}} (s)  -  \gamma_{\phi, \alpha_t } (s)  | \dint s  =   \int_0^\infty | \ \left( \phi^{\#}  \gamma_{\phi, \alpha_t } \right)_{\langle L \rangle} (y)  -  \phi^{\#} \gamma_{\phi,\alpha_t } (y)  | \dint y.
\end{align*}
Now by the second point in Lemma \ref{lem:tech}, there is exists an $L_2$ such that for every  $L \geq L_2$, 
\begin{align} \label{eq:vespers2}
\int_0^\infty | \ \left( \phi^{\#}  \gamma_{\phi, \alpha_t } \right)_{\langle L \rangle} (y)  -  \phi^{\#} \gamma_{\phi,\alpha_t } (y)  | \dint y \leq \rho/2k. 
\end{align}
In particular, combining \eqref{eq:vespers2} and \eqref{eq:vigil} with \eqref{eq:vespers1} in \eqref{eq:pastel}, we see that for all $L \geq L_0 := \max \{ L_1, L_2 \}$, 
\begin{align} \label{eq:funnelweb}
|\Delta_L| \leq \rho. 
\end{align}
Finally, since the truncated potential $\phi_L$ has a domain of finite Lebesgue measure, by Lemma \ref{lem:finite pull} we have
\begin{align} \label{eq:huntsman}
 \int_{ \mathbb{R}^k } \Bigg|  \mu_{\phi_L,t}^{N \to k} (s) - \gamma_{\phi, \alpha_{t,L}}^{\otimes k} (s)  \Bigg|  \dint s 
\begin{cases}
\left( 1 + \varepsilon^{(1)}_{k,N,L} \right) \xi_{k,\alpha_t}\frac{k}{N} \qquad &: \text{$k$ fixed, $N \to \infty$},\\
\left( 1 + \varepsilon^{(2)}_{k,N,L}  \right) \xi \frac{k}{N} \qquad &:\text{$k = o(N), k,N \to \infty$},\\
\left( 1 + \varepsilon^{(3)}_{k,N,L}  \right) Q(\theta) \qquad &:\text{$k \sim \theta N, k,N \to \infty$},
\end{cases}
\end{align}
where there is a universal constant $C \in(0,\infty)$ such that if $X_L$ is distributed according to $ \left( \gamma_{\alpha_{t}} \right)_{  \langle L \rangle } $, then with $C_L' := C \frac{ \mathbb{E}[ (X_L - \mathbb{E}[X_L ] )^4 }{ \mathrm{Var}[X_L]^2 }$ we have 
\begin{align*}
|\varepsilon^{(1)}_{k,N,L}| \leq C'_L \sqrt{ \frac{k}{N} }, \qquad |\varepsilon^{(2)}_{k,N,L}|, |\varepsilon^{(3)}_{k,N,L}|  \leq C'_L \left( \frac{1}{\sqrt{k}} + \frac{1}{\sqrt{N-k}} \right).
\end{align*}
Now by applying the third point of Lemma \ref{lem:tech} to the truncated density $\left( \gamma_{\alpha_{t}} \right)_{  \langle L \rangle }$, we see that there exists a constant $M\in(0,\infty)$ depending on $\phi,t$ but independent of $L$ such that for all $L \geq L_0$ we have 
\begin{align*}
 \frac{ \mathbb{E}[ (X_L - \mathbb{E}[X_L ] )^4 }{ \mathrm{Var}[X_L]^2 }  \leq M,
\end{align*} 
so that for all $L \geq L_0$ by setting $C' = CM$ we have 
\begin{align} \label{eq:blackwidow}
|\varepsilon^{(1)}_{k,N,L}| \leq C' \sqrt{ \frac{k}{N} }, \qquad |\varepsilon^{(2)}_{k,N,L}|, |\varepsilon^{(3)}_{k,N,L}|  \leq C' \left( \frac{1}{\sqrt{k}} + \frac{1}{\sqrt{N-k}} \right).
\end{align}
In particular, since $\rho  > 0$ and $L$ are arbitrary, by combining \eqref{eq:tarantula}, \eqref{eq:funnelweb}, \eqref{eq:huntsman} and \eqref{eq:blackwidow} we obtain the result. 
\end{proof}

%%%%%%%%%%%%%%%%%%%%%%%%%%%
%%%%%%%%%%%%%%%%%%%%%%%%%%%
%%%%%%%%%%%%%%%%%%%%%%%%%%%
\section*{Appendix} \label{sec:qgibbs}
%%%%%%%%%%%%%%%%%%%%%%%%%%%
%%%%%%%%%%%%%%%%%%%%%%%%%%%
%%%%%%%%%%%%%%%%%%%%%%%%%%%

In this appendix we sketch a few details of the proof of Lemma \ref{lem:qgibbs}, which is very similar to the proof of Theorem 1.6 in \cite{DF2}.

\begin{proof}[Proof of Lemma \ref{lem:qgibbs}]
Since the proof runs almost identically to Section 3 of Diaconis and Freedman \cite{DF2}, we will only sketch the details in the case where $k \sim \theta N$. 

Now, whenever $\nu$ is a probability density on $\mathbb{R}^k$ of a random variable $(X_1,\ldots,X_k)$, let $\bar{\nu}$ denote the probability density on $\mathbb{R}^1$ of $X_1+ \ldots + X_k$. According to the sufficiency Lemma 2.4 of \cite{DF2}, we have the equality in total variation
\begin{align*}
|| \psi_{\geq t}^{N \to k}-  \psi_{\alpha_t}^{ \otimes k} || = || \bar{\psi}_{ \geq t}^{N \to k} - \bar{\psi}_{\alpha_t}^{ \otimes k} ||.
\end{align*}
Write $f_k(s) :=  \bar{\psi}_{\alpha_t}^{ \otimes k}(s)$. 
We remark that if $\psi^{\star k}(s)$ is the $k$-fold convolution of the density $\psi$, then
\begin{align} \label{eq:tiltchange}
f_k(s) = \frac{ e^{\alpha_t s} \psi^{\star k}(s) }{Z(\alpha_t)^k} .
\end{align}
Using \eqref{eq:tiltchange} to obtain the second equality below, and the definition of $\bar{\psi}_{\geq t}^{N \to k}$ to obtain the first, the density associated with $\bar{\psi}_{\geq t}^{N \to k}$ is given by 
\begin{align*}
\bar{\psi}_{\geq t}^{N \to k}(s) = \frac{ \psi^{ \star k}(s) \int_{t N - s}^\infty \psi^{ \star(N-k)} (u) \, \dint u }{ \int_{tN}^\infty \psi^{ \star N } (u) \, \dint u} =\frac{ f_k(s) \int_{t N - s}^\infty e^{ - \alpha( u - tN +s ) } f_{N-k}(u) \, \dint u }{ \int_{tN}^\infty e^{ - (u - tN) } f_N(u) \, \dint u}.
\end{align*}
In particular, we have 
\begin{align*}
 ||\bar{\psi}_{\geq t}^{N \to K}- \bar{\psi}_{\geq t}^{ \otimes k} || = \int_{ - \infty}^\infty f_k(s) \left| \frac{ \int_{t N - s}^\infty e^{ - \alpha( u - tN +s ) } f_{N-k} (u) \, \dint u }{ \int_{tN}^\infty e^{ - (u - tN) } f_N (u) \, \dint u} - 1  \right|\, \dint s
\end{align*} 
We now consider a recentering. Let $g_j$ denote the density of
\begin{align*}
\frac{ X_1 + \ldots + X_j -  j t }{ \sigma \sqrt{j}} 
\end{align*}
so that when $j$ is large, $g_j$ is close to the Gaussian density. Indeed, changing variable we have
\begin{align} \label{eq:TVexpansion}
 ||\bar{\psi}_{\geq t}^{N \to K}- \bar{\psi}_{\geq t}^{ \otimes k} || = \int_{ - \infty}^\infty g_k (s) \left| \frac{ \int_{ - \phi s }^\infty e^{ - \alpha \sigma \sqrt{N-k} \left( u + \phi s \right) } g_{N-k}(u) \, \dint u }{ \int_{0}^\infty e^{ - \alpha \sigma \sqrt{N} u  } g_N (u) \, \dint u} - 1  \right|\, \dint s,
\end{align} 
where $\phi = \sqrt{ \frac{ k}{ N-k} }$. We now use the Edgeworth expansion to estimate the ratio 
\begin{align} \label{eq:ratio}
\frac{ \int_{ - \phi s }^\infty e^{ - \alpha \sigma \sqrt{N-k} \left( u + \phi s \right) } g_{N-k}(u) \, \dint u }{ \int_{0}^\infty e^{ -  \alpha \sigma \sqrt{N} u  } g_N (u) \, \dint u}
\end{align}
as a function of the external variable $s$. Indeed, according to the Edgeworth expansion (see, e.g., \cite[Section XVI]{feller}) 
\begin{align} \label{eq:edgeworth}
g_j(s) = \frac{1}{ \sqrt{ 2\pi}}e^{ - s^2/ 2} \left( 1 +  \frac{\kappa(\alpha_t)}{ \sqrt{j} } (s^3 - 3s ) \right) + \varepsilon_j(s),
\end{align}
where for a random variable $X$ with density $h_{\alpha_t}$,  $\kappa(\alpha_t) := \mathbb{E} [ (X - \mathbb{E}[X])^3]/\mathrm{Var}[X]^{3/2}$ and there is a universal constant $C$ such that 
\begin{align*}
\Big| \varepsilon_j(s) \Big| \leq \frac{C \mathbb{E}[ (X - \mathbb{E}[X])^4]}{ \mathrm{Var}[X]^2  } \frac{1}{j}.
\end{align*}
We begin with estimating the ratio \eqref{eq:ratio} when $k = \theta N$. Here, it is easily seen that
\begin{align} \label{eq:ratio2}
\frac{ \int_{ - \phi s }^\infty e^{ - \alpha \sigma \sqrt{N-k} \left( u + \phi s \right) } g_{N-k}(u) \, \dint u }{ \int_{0}^\infty e^{ -  \alpha \sigma \sqrt{N} u  } g_N (u) \, \dint u} &= \frac{ \frac{1}{ \alpha_t \sigma_t \sqrt{ 2 \pi (N-k) } } \exp \left(  - \frac{1}{2} (\phi s)^2 \right) + O(\frac{1}{N-k})  }{  \frac{1}{ \alpha_t \sigma_t \sqrt{ 2 \pi N } } + O(1/N) } \nonumber \\ 
& = \frac{1}{ \sqrt{1-\theta}} \exp\left( - \frac{1}{2} \frac{ \theta}{1-\theta} s^2  \right) + \varepsilon_{k,N}(s),
\end{align}
where there is a universal constant $C\in(0,\infty)$ such that for all $s \in \mathbb{R}$ we have 
\begin{align*}
|\varepsilon_{k,N}(s)| \leq C  \frac{ \mathbb{E} [ | X - \mathbb{E}[X]|^4] }{ \mathrm{Var}[X]^{3/2} } \frac{1}{ \sqrt{N-k}} 
\end{align*}
 By plugging \eqref{eq:ratio2} into \eqref{eq:TVexpansion}, and using the central limit theorem, we prove the statement of Lemma \ref{lem:qgibbs} in the setting where $k \sim \theta N$. 

We omit the proof of the case $k = o(N)$, which is similar.
\end{proof}

\subsection*{Acknowledgment} 
SJ and JP have been supported by the Austrian Science Fund (FWF) Project P32405 ``Asymptotic Geometric Analysis and Applications'' of which JP is principal investigator. JP is also supported by the Special Research Program Project F5508-N26.

\bibliographystyle{plain}
\bibliography{qmp}

\bigskip
\bigskip
	
	\medskip
	
	\small

	\noindent \textsc{Samuel G.~G. Johnston:} Institute of Mathematics and Scientific Computing,
		University of Graz, Heinrichstrasse 36, 8010 Graz, Austria
		
		\noindent
		{\it E-mail:} \texttt{samuel.johnston@uni-graz.at}

		\medskip
	
	\noindent \textsc{Joscha Prochno:} Institute of Mathematics and Scientific Computing,
	University of Graz, Heinrichstrasse 36, 8010 Graz, Austria
	
	\noindent
	{\it E-mail:} \texttt{joscha.prochno@uni-graz.at}

\end{document}